\documentclass{amsart}
\usepackage[all, cmtip]{xy}

\usepackage{amssymb,amsmath,amscd,xy,graphicx,textcomp,mathtools}
\usepackage{epsf} 
\usepackage{titlepic}
\usepackage{hyperref}
\usepackage{tabulary}
\usepackage{booktabs}

\hypersetup{
  colorlinks   = true, 
  urlcolor     = blue, 
  linkcolor    = blue, 
  citecolor   = blue 
}

\newtheorem{theorem}{Theorem}[section]
\newtheorem{lemma}[theorem]{Lemma}

\newtheorem{definition}[theorem]{Definition}

\newtheorem{remark}[theorem]{\it Remark}
\newtheorem{example}[theorem]{Example}
\newtheorem{proposition}[theorem]{Proposition}

\xyoption{arrow}

\xyoption{matrix}

\def\C{\mathbb{C}}
\def\R{\mathbb{R}}
\def\Z{\mathbb{Z}}

\def\Q{\mathbb{Q}}

\def\tree{\mathcal{T}}

\title[Polyhedra for character varieties and configuration spaces]{Newton-Okounkov polyhedra for character varieties and configuration spaces}
\author{Christopher Manon} 

\begin{document}

\begin{abstract}
We construct families of Newton-Okounkov bodies for the free group character varieties and configuration spaces of any connected reductive group.  We use this construction to give a proof that these spaces are Cohen-Macaulay. 
\end{abstract}


\maketitle

\tableofcontents


\section{Introduction}

The Newton-Okounkov body construction was defined by Okounkov in \cite{Ok} to study Hilbert functions of spherical varieties. This construction has been reintroduced in the recent papers of Lazarsfeld, Musta$\c{t} \check{a}$ \cite{LM} and Kaveh, Khovanskii \cite{KK}, and has become the centerpiece of several efforts to further integrate algebra, geometry, and combinatorics. A Newton-Okounkov body is a convex set $C_v$ associated to a commutative domain $A$ of finite Krull dimension by taking the convex hull of the image of a maximal rank valuation $v: A \to \Z^M \subset \R^M.$    Similarly, a compact, convex Newton-Okounkov body $C_v(\mathcal{L})$ can be defined for a projective variety $X$ and a choice of ample line bundle $\mathcal{L}$, for both of these constructions see Section \ref{background}.   As their name suggests, the Newton-Okounkov body is a generalization (see \cite{K1}) of the Newton polyhedron of an affine or projective toric variety.   In keeping with the analogy to the toric case, the Newton-Okounkov body serves as a discrete model of a variety, and a useful tool to study combinatorial and algebraic properties of its coordinate ring.

 The image $v(A) \subset \Z^M$ of a maximal rank valuation contains $0$ and is closed under addition, making it into an affine semigroup. When $v(A)$ is finitely generated the Newton-Okounkov body is polyhedral and there is flat degeneration $Spec(A) \Rightarrow X_{C_v}$, where $X_{C_v}$ is the toric variety attached to $C_v$.   In this way, Newton-Okounkov bodies provide a mechanism to produce toric degenerations of varieties, these are a useful tool for studying algebraic presentations of $A$ and studying the singularities of $Spec(A)$ (see Theorem \ref{ACM} below).  In this case, Harada and Kaveh have linked Newton-Okounkov body to symplectic geometry by using the degeneration to define integrable systems in $X$ when it is smooth and projective, \cite{HK}.  The Newton-Okounkov body is also useful as a combinatorial tool, as the set $v(A) \subset C_v$ provides a polyhedral labelling of a basis of $A$ which can be brought to bear when the underlying vector space of $A$ has an enumerative meaning, indeed this was what motivated their introduction by Okounkov in \cite{Ok}.

Newton-Okounkov bodies are one of several ways valuations are used to attach polyhedral geometry to structures in commutative algebra and algebraic geometry. In one of the first results of this type, Morgan and Shalen \cite{MS} reproduce a mysterious piecewise-linear compactification of Teichm\"uller space defined by Thurston as a polyhedral complex of valuations on the coordinate ring of one of the spaces we study in this paper, a complex character variety.   The character variety $\mathcal{X}(\pi, G)$ associated to a finitely generated group $\pi$ and a connected reductive group $G$  is the moduli space of representations $\rho:\pi \to G$.   When $\pi$ is the fundamental group of a smooth manifold $M$, $\mathcal{X}(\pi, G)$ is the moduli space of flat, topological principal $G$ bundles on $M$.  For this reason, character varieties appear as natural moduli spaces in subjects which blend topology and geometry, such as Teichm\"uller theory and the theory of geometric structures \cite{FG}, \cite{G1}, \cite{G2}.  We use combinatorial elements from representation theory to build Newton-Okounkov bodies for $\mathcal{X}(F_g, G)$, where $F_g$ is a free group.    

For a graph $\Gamma$, always assumed to be connected, we let $E(\Gamma)$ and $V(\Gamma)$ be the sets of edges and vertices, respectively.   Recall that a leaf is a vertex which is connected to precisely one edge, and a spanning tree $\tree \subset \Gamma$ is a subtree such that $V(\tree) = V(\Gamma).$  See Section \ref{backgroundreptheory} for the definition of the set $R(w_0)$ which appears below.

\begin{theorem}\label{character}
To the following information we associate a valuation $v_{\bold{i}, \Gamma}$ on $\C[\mathcal{X}(F_g, G)]$ with a convex, polyhedral Newton-Okounkov body $C_{\bold{i}}(\Gamma)$:

\begin{enumerate}
\item A trivalent graph $\Gamma$ with no leaves and $\beta_1(\Gamma) = g$.\\
\item Total orderings on the edges $E(\Gamma)$ and vertices $V(\Gamma)$ of $\Gamma.$\\
\item A spanning tree $\tree \subset \Gamma.$\\
\item An orientation on each edge in the set $\vec{e} = E(\Gamma) \setminus E(\tree)$\\
\item An assignment $\bold{i}: V(\Gamma) \to R(w_0)$ of reduced decomposition of the longest word $w_0$ in the Weyl group of $G$ to each vertex $v \in V(\Gamma)$.\\
\end{enumerate}

\end{theorem}

The methods we use to prove Theorem \ref{character} apply with little modification to another class of varieties, the configuration spaces $P_{\vec{\lambda}}(G)$ of points on flag varieties of $G$. For what follows see Section \ref{backgroundreptheory} and the book of Fulton and Harris \cite{FH} for the relevant background in representation theory, and the books of Mumford, Fogarty and Kirwan \cite{MFK} and Dolgachev \cite{D} for the necessary background on Geometric Invariant Theory $(GIT)$.  From now on $\vec{\lambda} = \lambda_1, \ldots, \lambda_n \in \Delta$ denotes a tuple of dominant weights of $G$ with associated highest weight representations $V(\lambda_1), \ldots, V(\lambda_n)$.  A central problem of combinatorial representation theory is to count the dimension of the space of invariant tensors:

\begin{equation}
(V(\lambda_1) \otimes \ldots \otimes V(\lambda_n))^G \subset V(\lambda_1) \otimes \ldots \otimes V(\lambda_n)\\
\end{equation}

Let $P_1, \ldots, P_n$ be the parabolic subgroups which stabilize the highest weight vectors in the representations $V(\lambda_1), \ldots, V(\lambda_n),$ respectively.  Recall that the flag variety $G/P_i$ has a $G-$linearized line bundle $\mathcal{L}_{\lambda_i}$, with $H^0(G/P_i, \mathcal{L}_{\lambda_i})$ equal to the dual repesentation $V(\lambda_i^*).$ These are the ingredients needed to define $P_{\vec{\lambda}}(G)$ as the following projective $GIT$ quotient:  
 
\begin{equation}
P_{\vec{\lambda}}(G) = G \backslash_{\mathcal{L}_{\lambda_1} \boxtimes \ldots \boxtimes \mathcal{L}_{\lambda_n}} \prod G/P_i\\
\end{equation}

\noindent
This construction turns the problem of counting the dimension of tensor product invariant spaces into the problem of finding the dimension of the space of global sections $H^0(P_{\vec{\lambda}}(G), \mathcal{L}(\vec{\lambda}))$ of the induced line bundle on $P_{\vec{\lambda}}(G).$

To find a solution to this problem, it suffices to find a rule to count the triple tensor product multiplicity spaces $(V(\lambda) \otimes V(\eta) \otimes V(\mu))^G$ (see Section \ref{step2} ).  In type $A$, this is accomplished by various combinatorial rules: (ordered by increasing generality) the Clebsh-Gordon rule for $G = SL_2(\C)$, the Pieri rule for $SL_m(\C)$ when the dominant weight $\lambda$ is a multiple $r_1\omega_1$  of the first fundamental weight, and the Littlewood-Richardson rule for general $SL_m(\C)$ triple tensor product multiplicities.   As a common feature, all of these rules can be realized by assigning to $\lambda, \eta, \mu$ the set of integer points in certain polytopes $P(\lambda, \eta, \mu).$ 

Families of polyhedral counting rules were discovered for general semisimple $G$ by Berenstein and Zelevinksy in \cite{BZ2} (their result applies readily to general connected, reductive $G$). Their proof utilizes a combination of positive geometry and the Lusztig's dual canonical basis.  Our second main theorem realizes the polyhedra of Berenstein and Zelevinsky, along with their generalizations to $n-$fold tensor products, as Newton-Okounkov bodies for the line bundles $\mathcal{L}(\vec{\lambda})$ on the $P_{\vec{\lambda}}(G).$ 

For a tree $\tree,$ we let $L(\tree)$ be the set of leaves, and we let $V^o(\tree) = V(\tree) \setminus L(\tree).$ Similarly, we let $E^o(\tree) \subset E(\tree)$ denote the set of edges which are not connected to a leaf.

\begin{theorem}\label{configuration}
To the following information we associate a polytope $C_{\bold{i}}(\tree, \vec{\lambda^*}),$ realized as the Newton-Okounkov body of a valuation $v_{\bold{i}, \tree}$ on the projective coordinate ring $\C[P_{\vec{\lambda}}(G)] =$ $\bigoplus_{m \geq 0} H^0(P_{\vec{\lambda}}(G), \mathcal{L}(\vec{\lambda})^{\otimes m})$:

\begin{enumerate} 
\item A trivalent tree $\tree$ with an ordering on the set of  leaves $L(\tree)$.\\
\item A total ordering on the non-leaf edges $E^o(\tree)$ and non-leaf vertices $V^o(\tree).$\\
\item An assignment $\bold{i}: V^o(\tree) \to R(w_0)$ of reduced decomposition of the longest word $w_0$ in the Weyl group of $G$ to each vertex $v \in V^o(\tree)$.\\
\end{enumerate}

In particular, the integer points of $C_{\bold{i}}(\tree, \vec{\lambda^*})$ are in bijection with a basis of the invariant tensors $(V(\lambda_1^*) \otimes \ldots \otimes V(\lambda_n^*))^G = H^0(P_{\vec{\lambda}}(G), \mathcal{L}(\vec{\lambda})).$

\end{theorem}

The $C_{\bold{i}}(\tree, \vec{\lambda})$ are cross-sections of an affine Newton-Okounkov body $C_{\bold{i}}(\tree)$ for an affine configuration space $P_n(G),$ defined in Section \ref{step1} (Example \ref{configurationspaces}).  This space acts as a ``master'' space for the configuration spaces, in that any $P_{\vec{\lambda}}(G)$ is obtained from $P_n(G)$ by a right $T^n$ $GIT$ quotient.  Using the same methods as in the proof of Theorem \ref{character}, we produce a $T^n$ invariant valuation $v_{\bold{i}, \tree}$ on $\C[P_n(G)]$ with Newton-Okounkov body $C_{\bold{i}}(\tree)$.

As an application of Theorems \ref{character} and \ref{configuration}, we show that a common structural property is shared by the two families of commutative algebras we consider.   Recall that a local ring $(R, m)$ is said to be Cohen-Macaulay  (CM) when its dimension coincides with its depth, see Chapter $2$ of \cite{BH}.   An affine variety is then Cohen-Macaulay if all of its local rings are CM.  Likewise, a projective variety with a chosen embedding is said to be Arithmetically Cohen-Macaulay if its projective coordinate ring is CM. 

\begin{theorem}\label{ACM}
For $G$ a connected, reductive group over $\C$, the affine varieties $\mathcal{X}(F_g, G)$ and $P_n(G)$ are Cohen-Macaulay.   Moreover, the projective varieties $P_{\vec{\lambda}}(G)$ are arithmetically Cohen-Macaulay. 
\end{theorem}

 Theorem \ref{ACM} follows from Proposition \ref{normalsemigroup}, which shows that the affine semigroups obtained as the images of the coordinate rings $\C[\mathcal{X}(F_g, G)], \C[P_n(G)]$ and the projective coordinate ring $\C[P_{\vec{\lambda}^*}(G)]$ under the valuations $v_{\bold{i}, \tree}, v_{\bold{i}, \Gamma}$ are normal.   Normal affine semigroup algebras are known to be Cohen-Macaulay by a result of Hochster, (see $6.3$ of \cite{BH}), and this property is preserved under flat degenerations.

\subsection{Methods}

We construct the valuations $v_{\bold{i}, \Gamma}$ by building increasing filtrations of the coordinate ring $\C[\mathcal{X}(F_g, G)]$ by a lexicographically ordered free Abelian group in two steps, given in Sections \ref{step1} and \ref{step2}.  These filtrations are shown to be ``strong'' filtrations (see Section \ref{background}), which are naturally associated to valuations by Proposition \ref{equivalenceprop}. 

In Section \ref{step1} we build a filtration inspired from one of the applications of  character varieties to gauge theory.  For a maximal compact subgroup $K \subset G$, so-called $BF$ theory on an appropriately triangulated manifold $M$ is quantized by combinatorial objects known as the spin diagrams of $K$ (equivalently $G$), see \cite{Baez}.

\begin{definition}
Let $\Gamma$ be an oriented graph, a spin diagram with topology $\Gamma$
is the following information:

\begin{enumerate}
\item An assignment $\lambda : E(\Gamma) \to \Delta,$ of dominant weights to the edges of $\Gamma.$\\
\item An assignment of $G-$linear maps $\rho$ to the vertices $v \in V(\Gamma)$ which intertwine the incoming representations $\bigotimes_{e \to v} V(\lambda(e))$ with the outgoing representations $\bigotimes_{v \to f} V(\lambda(f))$.\\
\end{enumerate}

\end{definition}

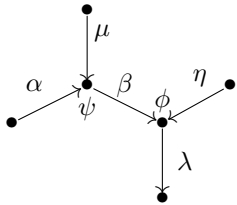
\begin{figure}[htbp]

$$
\begin{xy}
(0, 0)*{\bullet} = "A1";
(0, 10)*{\bullet} = "A2";
(0, 13)*{\phi};
(9, 15)*{\bullet} = "A3";
(-10, 15)*{\bullet} = "A4";
(-10, 12)*{\psi};
(-10, 25)*{\bullet} = "A5";
(-20, 10)*{\bullet} = "A6";
(3, 5)*{\lambda};
(5, 16)*{\eta};
(-8, 22)*{\mu};
(-17, 15)*{\alpha};
(-5, 15)*{\beta};
"A1"; "A2";**\dir{-}? >* \dir{>};
"A2"; "A3";**\dir{-}? >* \dir{>};
"A2"; "A4";**\dir{-}? >* \dir{>};
"A4"; "A5";**\dir{-}? >* \dir{>};
"A4"; "A6";**\dir{-}? >* \dir{>};
\end{xy}
$$\\
\caption{A spin diagram on a $4-$tree.}
\label{spind
}
\end{figure}

The purpose of Section \ref{step1} is to show that for a fixed trivalent $\Gamma$ with $\beta_1(\Gamma) = g$, the spin diagrams with topology $\Gamma$ define a filtration of $\C[\mathcal{X}(F_g, G)]$.  We identify the associated graded algebra of this filtration, and note that it is not an affine semigroup algebra unless the semisimple part of $G$ has a product of copies of $SL_2(\C)$ as its universal cover.

In order to enhance these filtrations, we must carefully choose a basis with amenable multiplication and combinatorial properties in the intertwiner spaces at each vertex $v \in V(\Gamma).$  This is provided by the dual canonical basis constructed by Lusztig, \cite{Lu}. The dual canonical basis can be used to define a basis in each invariant space $B(\mu, \lambda, \eta) \subset (V(\mu) \otimes V(\lambda) \otimes V(\eta))^G$, which are in turn identified with intertwiner spaces, \cite{BZ1}, \cite{BZ2}. We study filtrations built from this basis in Section \ref{step2}. 

For each choice $\bold{i} \in R(w_0)$ of a reduced decomposition of the longest element of the Weyl group of $G$, there is a labelling of the dual canonical basis by tuples of non-negative integers $b \to \vec{t} \in \Z^N$ called string parameters. In this way, the choice $\bold{i}$ assigns the elements of $(V(\mu)\otimes V(\lambda) \otimes V(\eta))^G$ to integer points in a convex polytope $C_{\bold{i}}(\mu, \lambda, \eta)$ studied in \cite{BZ2}.  We use the inequalities of these polytopes to define the polyhedra in Theorems \ref{character} and \ref{configuration}. 

Our use of the dual canonical basis in this role follows previous work of Caldero \cite{C}, and Alexeev, Brion \cite{AB} (see also \cite{K}), who use a filtration on the string parameters of the dual canonical basis to define toric degenerations of schubert varieties and spherical varieties, respectively.  We combine the filtrations from Section \ref{step1} with the string parameter filtrations of Section \ref{step2} to produce maximal rank valuations on $\C[\mathcal{X}(F_g, G)]$ and $\C[P_n(G)]$ with Proposition \ref{compositeprop} in Section \ref{background}, this proves Theorems \ref{character} and \ref{configuration}.

\subsection{Remarks}

Lawton \cite{La} and Sikora \cite{S} have given structure theorems for the coordinate rings $\C[\mathcal{X}(F_g, G)]$, and Lawton, Florentino give descriptions of the topology \cite{FL2} and the singular locus \cite{FL1} of $\mathcal{X}(F_g, G)$ in certain cases.  It would be interesting to relate the degenerations constructed here to a Gr\"obner theory of their defining equations. 

Theorem \ref{configuration} gives a construction of a basis for the tensor product invariant spaces $(V(\lambda_1)\otimes \ldots \otimes V(\lambda_n))^G$ which is labelled by the lattice points in a convex, rational polytope.   Howe, Jackson, Lee, Tan, and Willenbring \cite{HJLT}, \cite{HTW} use a SAGBI construction to achieve this for tensor product invariant spaces in the case $G = SL_m(\C).$  The cone $\Omega_3$ resulting from their construction on triple tensor products is a slice of the cone of Gel'fand-Tsetlin patterns, and is linearly equivalent to $C_{\bold{i}}(3)$ for a particular choice of decomposition $\bold{i}$. The algebraic structure of these cones is not very well understood outside the cases $SL_m(\C)$, $m = 2, 3, 4$.  We also point out that the space $P_n(SL_2(\C))$ is the affine cone of the Pl\"ucker embedding of the Grassmannian variety $Gr_2(\C^n)$, and that the toric degenerations we construct in this case coincide with those  found by Speyer and Sturmfels in \cite{SpSt}.  
 
Other enumeration problems in representation theory could plausibly be studied with the methods in this paper.  Polyhedra which control Levi branching problems $L \subset G$  have been defined by Berenstein and Zelevinsky in \cite{BZ2}.  These can be adapted along the lines of the program used in Sections \ref{step1}, \ref{step2}, and realized as Newton-Okounkov bodies.  

The definition of Newton-Okounkov body we use is more general than the one in \cite{LM} and \cite{KK}, where the valuation used to construct the Newton-Okounkov body comes from a flag of subspaces of the variety.  It would be interesting to realize the tensor product polytope $C_{\bold{i}}(\tree, \vec{\lambda})$ as the Newton-Okounkov body attached to a flag $\mathcal{F}$ in a variety birational to $P_{\vec{\lambda}}(G)$. This has been carried out in \cite{M16} for the valuations we construct here for the space $\mathcal{X}(F_g, SL_2(\C))$.

The work of Harada and Kaveh \cite{HK} suggests that each $C_{\bold{i}}(\tree, \vec{\lambda})$ and $C_{\bold{i}}(\Gamma)$ should be the momentum image of an integrable system in $P_{\vec{\lambda}^*}(G)$ and $\mathcal{X}(F_g, G),$ respectively.   A construction of such an integrable system for each of these polyhedra would be interesting for the symplectic geometry of $\mathcal{X}(F_g, G)$ and $P_{\vec{\lambda}^*}(G)$. It would also be interesting to see geometric relationships between the integrable systems associated to different valuations given by our construction. Partial results in this direction appear in \cite{HMM} and \cite{M16} for $G = SL_2(\C).$  

The isomorphism $\Phi_{\tree, \vec{e}}$ in Proposition \ref{spantreeiso} can be precomposed with the automorphism of $\C[\mathcal{X}(F_g, G)]$ induced by any outer automorphism of the free group $\alpha \in Out(F_g)$ to produce a distinct set of maximal rank valuations. In \cite{M16} we explore this large family of Newton-Okounkov bodies in the case $G = SL_2(\C)$, and its connections with the outer space construction of Culler and Vogtmann \cite{CV}.  Presumably similar spaces could be defined for all reductive $G$ using the results of this paper. 

Finally, we remark that Theorem \ref{configuration} essentially appears in the unplublished
notes \cite{M5}, along with other remarks on the use of valuations in the study of branching problems.

\subsection{Acknowledgements}

We thank Kiumars Kaveh for numerous helpful conversations about Newton-Okounkov bodies. We also thank Sean Lawton and Adam Sikora for useful conversations about character varieties. We thank the reviewer for useful suggestions to improve the exposition.

\section{Background on representation theory}\label{backgroundreptheory}

We introduce the necessary background on reductive groups and their Lie algebras.  We refer the reader to the books of Fulton and Harris \cite{FH}, Dolgachev \cite{D}, and Grosshans \cite{Gr} for structural properties of reductive groups, as well as their actions on commutative algebras. 

 We let $G$ be a connected, complex reductive group with center $Z$, maximal torus $T$ and root system $R.$  The Lie algebras of $Z$ and $T$ are denoted by $\mathfrak{z}$ and $\mathfrak{t}$, respectively.  We choose a set of positive roots $R_+ \subset R \subset \mathfrak{t}^*$ with simple roots $S = \{\alpha_1, \ldots, \alpha_r\}.$

 The Lie algebra $\mathfrak{g}$ of $G$ splits as a direct sum $\mathfrak{g} = \mathfrak{z} \oplus \mathfrak{g}^{ss}$ with $\mathfrak{z}$ the Lie algebra of $Z$ and $\mathfrak{g}^{ss}$ the semisimple part of $\mathfrak{g}.$   The Lie algebra $\mathfrak{g}^{ss}$ further decomposes into a sum of simple Lie algebras $\mathfrak{g}^{ss} = \bigoplus \mathfrak{g}_i.$  We fix a system of Chevallay generators $H_1, \ldots, H_r$, $f_1, \ldots, f_r$, $e_1, \ldots, e_r \in \mathfrak{g}^{ss}$, where $r$ is the rank of $\mathfrak{g}^{ss}$.  The elements $H_i$ are the simple coroots, they form a basis of a Cartan subalgebra $\mathfrak{h} \subset \mathfrak{g}^{ss}$, and together with $\mathfrak{z}$ they span the Lie algebra $\mathfrak{t}$. The elements $f_i$ and $e_i$ are called lowering and raising operators, and they generate nilpotent Lie subalgebras $\mathfrak{u}_-, \mathfrak{u}_+ \subset \mathfrak{g},$ respectively.   The set $\{f_i, H_i, e_i\}$ generates a copy of of $sl_2(\C) \subset \mathfrak{g}$, and for any $H \in \mathfrak{t}$, $[H, f_i] = -\alpha_i(H)f_i,$ $[H, e_i] = \alpha_i(H)e_i$.  All of this information is captured in the Cartan matrix $A = [a_{ij}]$, with $a_{ij} = \alpha_j(H_i).$

The above information defines a distinguished so-called triangular decomposition $\mathfrak{g} = \mathfrak{u}_- \oplus \mathfrak{t} \oplus \mathfrak{u}_+$.  The spaces $\mathfrak{u}_-, \mathfrak{u}_+$ are the Lie algebras of maximal unipotent subgroups $U_-, U_+ \subset G$, which together with $T$ define a dense, open subset $U_- \times T \times U_+ \to U_-TU_+ \subset G,$.  When it is important
to single out one of these groups we use the notation $U_{\pm}$, otherwise we will refer to a general maximal unipotent subgroup $U$. 

We let $\Lambda \subset \mathfrak{t}^*$ be the character group of $T$. The weights $\lambda \in \Lambda$ which satisfy $\lambda(H_i) \geq 0$ are called dominant, and the non-negative real span of the dominant weights is called a Weyl chamber $\Delta \subset \mathfrak{t}^*$.   For $G^{ss}$ the simply-connected semisimple group associated to $\mathfrak{g}^{ss}$, we can present $G$ as $[G^{ss} \times L]/K$ for $L$ a torus and $K$ a finite central subgroup of $G \times L.$  In particular, any dominant weight $\lambda \in \Delta$ is a sum $\bar{\lambda} + \tau$ of a dominant weight $\bar{\lambda} \in \Delta^{ss}$ and a character $\tau$ of $L$ such that $\bar{\lambda} + \tau$ is a trivial character on $K$.  We choose one such presentation for $G$, and a corresponding isomorphism of the Lie algebra of $L$ with $\mathfrak{z}$. In this way, the Weyl chamber can be identified with a product $\mathfrak{z}^* \times \prod \Delta_i$, where $\Delta_i$ is a Weyl chamber of a simple part $\mathfrak{g}_i$ of $\mathfrak{g}^{ss}.$  The vectors $\omega_i \in \Delta^{ss}$ which satisfy $\lambda_j(H_i) = \delta_{ij}$ are called the fundamental weights, these are weights of $G$ when $G$ is simply-connected.   There is a natural partial ordering $<$ on the set of dominant weights, where $\lambda > \eta$ when $\lambda - \eta$ can be expressed as a non-negative sum of positive roots. 

The Weyl group $N(T)/T = W$ of $G$ is generated by elements $s_1, \ldots, s_r$ called simple reflections, these are in bijection with the simple roots.  The simple reflections define a length function on the elements of $W.$  The unique longest element $w_0$ has length equal to the number $N$ of positive roots of $\mathfrak{g}$.   We let $R(w_0)$ denote the set of reduced decompositions $\bold{i} = \{i_1, \ldots, i_N\}$, of $w_0 = s_{i_1}\ldots s_{i_N}$ into simple reflections.

The category $Rep(G)$ of finite dimensional Representations of $G$ is semisimple, and the set of irreducible representations is in bijection with the dominant weights $\lambda \in \Delta.$  We let $V(\lambda)$ denote the irreducible representation associated to $\lambda \in \Delta$.   The dual representation $V(\lambda)^*$ of an irreducible has weight $\lambda^* = -w_0(\lambda).$  Recall that $V(\lambda)$ decomposes as a $T-$representation into weight spaces $V_{\eta}(\lambda) \subset V(\lambda)$, that $\eta < \lambda$ for all $\eta$ which appear in this decomposition, and that the weight space  $V_{\lambda}(\lambda)$ is always $1-$dimensional.  We let $b_{\lambda}$ be a highest weight vector spanning $V_{\lambda}(\lambda)$.

\section{Background on valuations and Newton-Okounkov bodies}\label{background}

We introduce basic concepts and operations on valuations, for more on this topic, see Section $1$ of \cite{KK} and Chapter $6$ of the book by Zariski and Samuel \cite{ZS}.   Let $A$ be a commutative domain over $\C$ of finite Krull dimension $dim(A),$ and let $\mathbb{G}$ be an Abelian group with a total ordering $\prec$. A valuation $\mathfrak{v}: A \setminus \{0\} \to \mathbb{G}$ is a function which satisfies $\mathfrak{v}(fg) = \mathfrak{v}(f) + \mathfrak{v}(g)$ and $\mathfrak{v}(f + g) \preceq max\{\mathfrak{v}(f), \mathfrak{v}(g)\}$ for all $f, g \in A.$ Additionally, we assume that $\mathfrak{v}(C)$ is equal to the neutral element $0 \in \mathbb{G}$ for all $C \in \C \setminus \{0\}$, this means that $\mathfrak{v}$ is a lift of the trivial valutaion on $\C$ to $A.$ 

For our purposes, the group $\mathbb{G}$ is always taken to be a free Abelian group $\Z^M$ in some guise, and $\prec$ is the lexicographic ordering. Recall that this means that $(a_1, \ldots, a_M) \prec (b_1, \ldots, b_M)$ if and only if the first nonzero $b_i - a_i$ is positive. 
The properties of a valuation imply that the image $\mathfrak{v}(A) \subset \Z^M$ of $\mathfrak{v}$ contains $0$ and is closed under addition, making it an affine semigroup. The rank of $\mathfrak{v}$ is the dimension of vector space spanned by $\mathfrak{v}(A)$ in $\Q^m,$ or equivalently the rank of the free Abelian subgroup generated by $\mathfrak{v}(A) \subset \Z^M.$  The Krull dimension of the affine semigroup algebra $\C[\mathfrak{v}(A)]$ is then equal to the rank of $\mathfrak{v}$, see Chapter $6$ of \cite{BH}.

\begin{definition}$($Newton-Okounkov Body: affine case$)$
With $A, \mathfrak{v}, \Z^M, \prec$ as above  and $rank(\mathfrak{v}) = dim(A)$, the Newton-Okounkov body $C_{\mathfrak{v}} \subset \R^m$ is the positive real cone $\R_{\geq 0}\mathfrak{v}(A)$ spanned by the image $\mathfrak{v}(A) \subset \Z^M \subset \R^M.$   
\end{definition}

\noindent
It follows that $dim(C_{\mathfrak{v}})$ as a subspace of $\R^M$ is equal to the dimension of the real space $\R \mathfrak{v}(A)$, which in turn equals $dim(A) = rank(\mathfrak{v})$.

 Let $\cup_{w \in \Z^M} F_{\preceq w} = A$ be an increasing algebra filtration of $A$ by $\C$ vector spaces.  We say the spaces $F_{\preceq w}$ define a strong filtration if $\C \subset F_{\preceq 0}$, $\C \not\subset F_{\preceq w}$ for any $w \prec 0$, and the associated graded algebra $gr_F(A) = \bigoplus_{w \in \Z^M} F_{\preceq w}/F_{\prec w}$ is a domain. 

\begin{proposition}\label{equivalenceprop}
Let $A$, $\Z^M, \prec$ and $F$ be as above. The information of a strong increasing filtration $A = \cup_{w \in \Z^M} F_{\leq w}$ is equivalent
to a valuation $\mathfrak{v}_F: A \to \Z^M$.
\end{proposition}

\begin{proof}
Starting with a filtration $F$, define $\mathfrak{v}_F$ by $\mathfrak{v}_F(f) = min\{w | f \in F_{\leq w}\}.$ For a valuation $\mathfrak{v}$ define $F^{\mathfrak{v}}_w \subset A$ by $F^{\mathfrak{v}}_w = \{f | v(f) \leq w\}.$ The property $\mathfrak{v}(fg) = \mathfrak{v}(f) + \mathfrak{v}(g)$ implies that $F^{\mathfrak{v}}$ is a strong filtration.  Similarly, $F$ being a strong algebra filtration implies that $\mathfrak{v}_F(fg) = \mathfrak{v}_F(f) + \mathfrak{v}_F(g).$ We leave
it to the reader to check the rest. 
\end{proof}

The following proposition allows us to construct valuations on $A$ in steps. 

\begin{proposition}\label{compositeprop}
Let $A$ be as above, with $F$ a strong filtration on $A$ by $(\Z^M, <_1),$ and let $G$
be a strong filtration on $gr_F(A)$ by $(\Z^L, <_2)$ which is compatible with the induced grading.   There is a strong filtration $F\circ G$ on $A$ by $(\Z^{M+L}, <_1\circ <_2),$ where $<_1\circ <_2$ is the composite order built lexicographically by first ordering by $<_1$ and breaking ties with $<_2.$ This filtration has associated graded algebra $gr_G(gr_F(A)).$
\end{proposition}

\begin{proof}
Each space $F_{\leq w}/F_{< w} \subset gr_F(A)$ has a filtration $\ldots \subset G_{w, u} \subset \ldots.$  We pull the spaces $G_{w, u}$ back to a filtration $F\circ G_{w, u} \subset F_w$.
By construction, each space $F \circ G_{w, u}$ contains $F_{< w}$, this implies that $F \circ G_{w', u'} \subset F\circ G_{w, u}$ if $w' < w.$  If $w' = w$, then $u' < u$ and $F\circ G_{w', u'} \subset F\circ G_{w, u}$ by construction.  It is straightforward to check the identity $gr_{F\circ G} = gr_G(gr_F(A))$, which proves the strong filtration property. 
\end{proof} 

\begin{example}
For a polynomial ring $\C[x_1, \ldots, x_M],$ we can define maximal rank valuation $\mathfrak{d}: \C[x_1, \ldots, x_M] \to \Z^M$
by sending a polynomial $p(\vec{x}) = \sum_{\beta \in \Z^M} C_{\beta}\vec{x}^{\beta}$ to the multidegree $\alpha$ of its highest term 
under the lexicographic ordering.   The value semigroup $\mathfrak{d}(\C[x_1, \ldots, x_n])$ of this valuation is then the set of non-negative multidegrees $\Z_{\geq 0}^M \subset \Z^M,$ and the Newton-Okounkov body $C_{\mathfrak{d}}$ is the positive orthant $\R_{\geq 0}^M \subset \R^M.$
\end{example}

In the presence of a rational action on $A$ by a reductive group $G$, we say a strong filtration $F$ is $G-$stable if each space $F_{\preceq w}$ is likewise a $G-$representation. It is straightforward to verify that this is the case if and only if the associated valuation $\mathfrak{v}_F$ is invariant with respect to the $G-$action.  In this case, there is a strong filtration of the algebra of invariants $A^G$ by the invariant spaces $F_{\preceq w}^G = F_{\preceq w} \cap A^G$, and there is a natural induced rational $G-$action on the associated graded algebra $gr_F(A)$.

\begin{proposition}\label{Gassociatedgraded}
Let $F$ be a $G-$stable filtration on $A$, and let $F$ also denote the induced filtration on the algebra of invariants $A^G \subset A.$ We have the following isomorphism:

\begin{equation}
gr_F(A^G) \cong gr_F(A)^G.\\
\end{equation}

\end{proposition}
  
\begin{proof}
This is a consequence of the reductivity of $G$, as it ensures that the following sequence remains exact after taking $G-$invariants:

\begin{equation}
0 \to F_{\prec w} \to F_{\preceq w} \to F_{\preceq w}/F_{\prec w} \to 0.\\
\end{equation}

\end{proof}

Let $A = \bigoplus_{L \geq 0} A_L$ be a graded domain of finite Krull dimension with a
maximal rank valuation $\mathfrak{v}: A \to \Z^M$.  Suppose further that one component of the function$\mathfrak{v}$ returns the grading, i.e. if $f \in A_L$ then $\mathfrak{v}(f) = (\ldots, L, \ldots)$.  In this case the Newton-Okounkov body $C_{\mathfrak{v}}$ carries a natural projection $\pi:C_{\mathfrak{v}} \to \R_{\geq 0}$ and is a cone over the fiber $\pi^{-1}(1).$  

\begin{definition}
With $A,$ $\mathfrak{v}$ as above, we let $\bar{C}_{\mathfrak{v}} = \pi^{-1}(1)$, this is a compact, convex set.  When $A$ is the projective coordinate ring of a projective variety $X$ with respect to an ample line bundle $\mathcal{L}$, we denote this set $\bar{C}_{\mathfrak{v}}(\mathcal{L})$, and refer to it as the Newton-Okounkov body of $\mathcal{L}.$ 
\end{definition}

\section{Branching filtrations}\label{step1}

We make use of the ordering on dominant weights to construct filtrations of the coordinate rings of character varieties and configuration spaces.  These filtrations are not fine enough to give affine semigroup associated graded algebras, however this construction reduces the problem to constructing a $T^3-$stable filtration on $\C[P_3(G)]$. Spin diagrams for the group $G$ emerge from this construction as labels for the graded components of the associated graded algebras we construct.  We also give an alternative $GIT$ construction of the character variety $\mathcal{X}(F_g, G)$ which makes the connection with spin diagrams more transparent.

\subsection{Horospherical contraction and the algebra $\C[G]$}

We briefly review the theory of horospherical contraction, due to Popov \cite{Po}. We choose a maximal torus $T \subset G$, with triangular decomposition $U_- T U_+ \subset G,$ and Weyl chamber $\Delta.$  The coordinate ring $\C[G]$ has the following isotypical decomposition as a $G\times G$ representation (for proof see e.g. Theorem $12.9$ of \cite{Gr}):

\begin{equation}
\C[G] = \bigoplus_{\lambda \in \Delta} V(\lambda) \otimes V(\lambda^*).\\
\end{equation}

Horospherical contraction relates $G$ to the affine right $GIT$ quotient $G/U$ of $G$ by any maximal unipotent $U \subset G$.  This scheme has coordinate ring $\C[G/U] = \bigoplus_{\lambda \in \Delta} V(\lambda) \otimes \C b_{\lambda^*} \subset \C[G]$ with multiplication computed by dualizing the map $C: V(\lambda + \eta) \to V(\lambda) \otimes V(\eta)$ which sends the highest weight vector $b_{\lambda + \eta}$ to $b_{\lambda} \otimes b_{\eta}$. 

The highest weight ordering induces a $G\times G$-stable filtration on $\C[G]$ by 
the subspaces $F_{\leq \eta} = \bigoplus_{\gamma \leq \eta} V(\gamma) \otimes V(\gamma^*)$.  The next proposition gives the associated graded algebra of this filtration. 

\begin{proposition}\label{hcontract}$($Horospherical contraction$)$
The dominant weight filtration on $\C[G]$ induced by $\Delta$ has associated graded
algebra isomorphic to $\C[T\backslash (G/U_+ \times U_- \backslash G)],$ where the $T-$action is on the right of $G/U_+$ and the left of $U_- \backslash G$.  In particular, this $T-$action has  isotypical spaces $V(\lambda) \otimes V(\lambda^*) \subset \C[G/U_+ \times U_- \backslash G]$.   
\end{proposition}

\begin{proof}
This follows from Chapter $3,$ Section $15$ of \cite{Gr}.  
\end{proof}

As $\C[T \backslash (G/U_+ \times U_- \backslash G)]$ is a domain, any prolongation of the partial ordering on $\Delta$ to a total ordering which is compatible with addition of weights defines a strong filtration $\cup_{\lambda \in \Delta} F_{\leq \lambda} = \C[G]$, where $F_{\leq \lambda} = \bigoplus_{\eta \leq \lambda} V(\eta)\otimes V(\eta^*).$  There are many such prolongations, we define one below. 

\begin{definition}
Let $G$ be a simple complex group, with Weyl chamber $\Delta,$ and simple coweights $H_1, \ldots, H_r$. 
The total order $\bold{<}$ is defined by lexicographically organizing the orderings defined by $\lambda(H_i) \in \Z$.
\end{definition}

Recall that we have chosen a product decomposition $\Delta = \mathfrak{z}^* \times \prod \Delta_i$.  We can define $\bold{<}$ on $\Delta$ by ordering the $\mathfrak{g}_i$ and using the induced lexicographic organization of the orderings $\bold{<}_i.$  We then break ties with any lexicographic ordering on a basis of $\mathfrak{z}^*.$ The following is straightforward. 

\begin{lemma}\label{proh}
The total order $\bold{<}$ respects addition of weights, and refines the partial dominant weight ordering $<$.
\end{lemma}

\begin{example}
We consider the ordering $\bold{<}$ on the Weyl chamber $\Delta$ of $GL_3(\C).$  A dominant weight $\lambda \in \Delta$ is typically
represented as a tuple $(\lambda_1, \lambda_2, \lambda_3)$ satisfying $\lambda_1 \geq \lambda_2 \geq \lambda_3,$ and the fundamental coroots are the vectors $H_1 = (1, -1, 0)$ and $H_2 = (0, 1, -1)$.  A weight is seperated into its semisimple part (a dominant weight of $SL_3(\C)$) and its central part by writing $(\lambda_1, \lambda_2, \lambda_3) = (\lambda_1 - \lambda_3, \lambda_2 - \lambda_3, 0) + \lambda_3(1, 1, 1).$  Consider the weights $(5, 3, 2), (4, 2, 1)$.  We have $H_1(5, 3, 2) = 2 = H_1(4, 2, 1)$ and $H_2(5, 3, 2) = 1 = H_2(4, 2, 1)$, however $(5, 3, 2) = (3, 1, 0 ) + 2(1, 1, 1)$ and $(4, 2, 1) = (3, 1, 0) + (1, 1, 1)$.  This implies that $(5, 3, 2) \bold{>} (4, 2, 1)$ in the ordering.  
\end{example}

\begin{remark}\label{flagreduce}

For an irreducible representation $V(\lambda)$, the orbit $G[b_{\lambda}] \subset \mathbb{P}(V(\lambda))$ of the highest weight line is a flag variety $G/P$.  The embedding $G/P \to \mathbb{P}(V(\lambda))$ endows $G/P$ with a line bundle $\mathcal{L}_{\lambda}$ with global section space $H^0(G/P, \mathcal{L}_{\lambda}) = V(\lambda^*)$.  This linearization can be constructed using an affine $GIT$ quotient of $G/U$ by $T$ with character $\lambda$.

\begin{equation}
Proj(\bigoplus_{m \geq 0} H^0(G/P, \mathcal{L}_{\lambda}^{\otimes m}) = [G/U]/_{\lambda} T\\
\end{equation}

\end{remark}

\subsection{Branching algebras}

We apply horospherical contraction to obtain filtrations on the coordinate rings of a class of spaces $B(\phi)$ called branching varieties. There is one such variety for each map $\phi: H \to G$ of complex, connected reductive groups. The space $P_n(G)$ is recovered as the branching variety $B(\delta_n)$, where $\delta_n: G \to G^{n-1}$ is the diagonal embedding.  We choose maximal unipotent subgroups $U_H \subset H$, $U_G \subset G,$ and Weyl chambers $\Delta_H, \Delta_G,$ and define $B(\phi)$ as the following affine $GIT$ quotient. 

\begin{equation} 
B(\phi) = H \backslash [H/U_H \times G/U_G]\\
\end{equation}

\noindent
Here $H$ acts on $H/U_H \times G/U_G$ diagonally on the left, where the action on $G/U_G$ is defined through $\phi.$ The coordinate ring of $B(\phi)$ is graded by the multiplicity spaces $W(\mu, \lambda)$ of $H$ irreducible representations in the irreducible representations of $G,$ as branched over the map $\phi.$

\begin{equation}
V(\lambda) = \bigoplus_{\mu \in \Delta_H} W(\mu, \lambda) \otimes V(\mu)
\end{equation}

\begin{equation}
\C[B(\phi)] = \bigoplus_{\mu, \lambda \in \Delta_H \times \Delta_G} W(\mu, \lambda)\\
\end{equation}

\begin{example}\label{configurationspaces}
Let $\delta_n: G \to G^{n-1}$ be the diagonal map $g \to (g, \ldots, g)$.  The branching variety $B(\delta_n)$ is then a left diagonal $G$ quotient of $G/U^n$, we also knows this space as $P_n(G).$  The coordinate ring $\C[P_n(G)]$ is the following direct sum of invariant spaces. 

\begin{equation}
\C[P_n(G)] = \bigoplus_{\vec{\lambda} \in \Delta^n} (V(\lambda_1) \otimes \ldots \otimes V(\lambda_n))^G\\
\end{equation}

The space $P_n(G)$ carries a residual right hand side $T^n$ action.  By Remark \ref{flagreduce}, the affine $GIT$ quotient $P_n(G)/_{\vec{\lambda}} T^n$ defined
by a tuple of dominant weights $(\lambda_1, \ldots, \lambda_n) = \vec{\lambda}$
can be identified with the configuration space $P_{\vec{\lambda}}(G)$.  To see this, note
that the invariant spaces $(V(m\lambda_1) \otimes \ldots \otimes V(m\lambda_n))^G$ corresponding to the multiples $m\vec{\lambda}$ of the character defined by $\vec{\lambda}$ are precisely the $G$ invariants in the global section spaces of the powers of the line bundle $\mathcal{L}_{\lambda_1} \boxtimes \ldots \boxtimes \mathcal{L}_{\lambda_n}$ on $G/P_1\times \ldots \times G/P_n.$
\end{example}

The branching algebras $\C[B(\phi)]$ come with special filtrations defined by diagrams
in the category of connected reductive groups.  We let $\phi = \pi \circ \psi$ be a factorization of $\phi$ in this category.

$$
\begin{CD}
H @>\psi>> K @>\pi>> G\\
\end{CD}
$$

The map $\psi$ defines an action of $H$ on the reductive group $K$, and $\pi$ defines an action of $K$ on $G$.  Using these actions, we can identify $B(\phi)$ with the following $GIT$ quotient. 

\begin{equation}
B(\phi) = H \times K \backslash [H/U_H \times K \times G/U_G]\\
\end{equation}

Here $H$ acts diagonally on the left of $H/U_H$ and $K$, and $K$ acts on the right of $K$ and on the left side of $G/U_G.$  The space $K \backslash K \times G/U_G$ is isomorphic to $G/U_G$, and likewise the resulting action of $H$ on $H/U_G\times G/U_G$ is induced through $\phi = \pi \circ \psi.$ The direct sum decomposition $\C[K] = \bigoplus_{\eta \in \Delta_K} V(\eta)\otimes V(\eta^*)$ induces a decomposition of $\C[B(\phi)]$.

\begin{equation}
\C[B(\phi)] = \bigoplus_{\lambda, \eta, \mu \in \Delta_G, \Delta_K, \Delta_H} W(\mu, \eta) \otimes W(\eta, \lambda)\\
\end{equation}

\noindent
This defines a $T_H \times T_G$-stable filtration $F^{\psi, \pi}$ of $\C[B(\phi)]$ by the dominant weights $\eta \in \Delta_K.$ 

\begin{equation}
F^{\psi, \pi}_{\leq \eta} = \bigoplus_{\lambda, \gamma \leq \eta, \mu} W(\lambda, \gamma) \otimes W(\gamma, \mu)\\
\end{equation}

\begin{proposition}\label{degrep}
 The associated graded algebra of $F^{\psi, \pi}$ is the affine $GIT$ quotient $\C[B(\pi) \times B(\psi)/ T_K]$, where $T_K$ acts
on $\C[B(\phi)]\otimes \C[B(\psi)]$ with isotypical spaces $W(\lambda, \gamma)\otimes W(\gamma, \mu).$
\end{proposition}

\begin{proof}
The filtration $F^{\psi, \pi}$ is induced from the horospherical filtration on $\C[K]$.  The $K\times K$ stability of horospherical contraction and the remarks above imply that the associated graded algebra is the domain $\C[B(\pi) \times B(\psi)]^{T_K}.$ 
\end{proof}

\noindent
Notice that the associated graded algebra $\C[B(\pi) \times B(\psi)]^{T_K}$ has a residual algebraic action of $T_K.$

We finish this subsection by applying this construction to the diagonal map $\delta_n: G \to G^{n-1}.$ Let $\tree$ be a tree with two internal vertices and $n$ leaves labelled $0, \ldots, n-1.$ Let $k +1$ be the number of edges incident on the vertex connected to the $0$  leaf, and $m$ be the number of leaves connected to the other internal vertex. This structure defines a factorization $\delta_n = Id ^s \times \delta_m \times Id ^t \circ \delta_k: G \to G^{n-1},$ where $s + t = k-1$ (see Figure \ref{factortree} for an example).  Proposition \ref{degrep} implies there is a filtration $F^{\tree}$ on $\C[P_n(G)]$, with associated graded algebra $\C[P_m(G) \times P_k(G)]^T$.  

\begin{example}
The factorization of $\delta_4$ below is represented by the tree in Figure \ref{factortree}.

$$
\begin{CD}
G @>\delta_3>> (G \times G) @>\delta_3 \times Id>> ((G\times G)\times G)\\
\end{CD}
$$

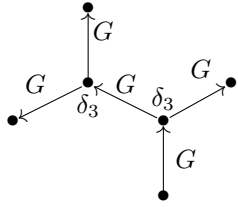
\begin{figure}[htbp]

$$
\begin{xy}
(0, 0)*{\bullet} = "A1";
(0, 10)*{\bullet} = "A2";
(0, 13)*{\delta_3};
(9, 15)*{\bullet} = "A3";
(-10, 15)*{\bullet} = "A4";
(-10, 12)*{\delta_3};
(-10, 25)*{\bullet} = "A5";
(-20, 10)*{\bullet} = "A6";
(3, 5)*{G};
(5, 17)*{G};
(-8, 22)*{G};
(-17, 15)*{G};
(-5, 15)*{G};
"A2"; "A1";**\dir{-}? >* \dir{>};
"A3"; "A2";**\dir{-}? >* \dir{>};
"A4"; "A2";**\dir{-}? >* \dir{>};
"A5"; "A4";**\dir{-}? >* \dir{>};
"A6"; "A4";**\dir{-}? >* \dir{>};
\end{xy}
$$\\
\caption{A directed tree of diagonal maps.}
\label{factortree}
\end{figure}
 
\end{example}

Given a trivalent tree $\tree$ with $n$ leaves, and an ordering on $E^o(\tree),$ we iterate this construction to obtain a filtration $F^{\tree}$ on $\C[P_n(G)]$.  

\begin{proposition}\label{treeweight}
For a trivalent tree $\tree$ with $n$ leaves, and an ordering on $E^o(\tree)$ there is a $T^n$-invariant filtration on $\C[P_n(G)]$ with associated graded algebra the coordinate ring of $P_{\tree}(G) = [\prod_{v \in V(\tree)} P_3(G)]/ T^{E(\tree)}.$ 
\end{proposition}

\begin{proof}
An ordering $E(\tree) = \{e_1, \ldots, e_{n-3}\}$ induces a length $n-3$ chain
of collapsing maps on trees, $\pi_i: \tree_{i-1} \to \tree_i$ where $\tree_0 = \tree$, 
and $\tree_i$ is obtained from $\tree_{i-1}$ via $\pi_i$ by collapsing the edge $e_i.$
The tree $\tree_{n-3}$ has a single internal vertex, and $\tree_{n-4}$ has a single internal edge. By the previous construction there is a $T^n$-stable filtration on $\C[P_n(G)]$ with associated graded algebra $P_{\tree_{n-4}}(G) = [P_{v(u)}(G) \times P_{v(w)}(G)]/T$, where $v(u)$ and $v(w)$ are the valences of the two internal vertices $u, w \in V(\tree_{n-4})$.  The map $\pi_{n-5}$ collapses the edge $e_{n-2}$ to either $u$ or $w$, yielding a corresponding filtration on $\C[P_{v(u)}(G)]$ or $\C[P_{v(w)}(G)].$  This filtration is invariant with respect to $T$ above, and so induces a filtration on $\C[P_{v(u)}(G) \times P_{v(w)}(G)]/T$.  We can now apply Proposition \ref{compositeprop} to obtain a filtration on 
$\C[P_n(G)]$. Continuing this way, we obtain the proposition. 
\end{proof}

\subsection{Character varieties and the master configuration space}

We show that a similar family of filtrations exist for the character variety $\mathcal{X}(F_g, G)$. This variety is constructed as the following $GIT$ quotient. 

\begin{equation}
\mathcal{X}(F_g, G) = G^g/_{ad} G\\
\end{equation}

The $ad$ subscript indicates the conjugation action of $G$ on the product $g \circ_{ad} (x_1, \ldots, x_n) =$ $(gx_1g^{-1}, \ldots, gx_ng^{-1}).$
The coordinate ring $\C[\mathcal{X}(F_g, G)]$ is therefore the algebra of conjugation invariants in $\C[G^g].$ By Proposition \ref{hcontract},
the horospherical contraction of $\C[G]$ to $\C[T \backslash (G/U_+ \times U_- \backslash G)]$ is $G\times G$ invariant, therefore we may define a filtration on  $\C[\mathcal{X}(F_g, G)]$ by placing this $G \times G$ stable filtration on the coordinate ring of each part $G$ of the product $G^g.$

\begin{proposition}\label{characterdegconfig}
There is a filtration on $\C[\mathcal{X}(F_g, G)]$ with associated graded ring equal to the coordinate ring of $[ T \backslash (G/U_+ \times U_- \backslash G)]^g /_{ad} G = P_{2g}(G)/T^g.$  Here the invariants of the torus $T^g$ are the tensor products $(V(\lambda_1)\otimes \ldots  \otimes V(\lambda_{2g}))^G \subset \C[P_{2g}(G)]$, where $\lambda_{2k-1}^* = \lambda_{2k}.$
\end{proposition}

Now that we have related the character variety $\mathcal{X}(F_g, G)$ to the master configuration space $P_{2g}(G),$ we may use the valuations constructed in Proposition \ref{treeweight}.  

\begin{proposition}\label{characterstep1}
For every choice of a trivalent graph $\Gamma$, spanning tree $\tree \subset \Gamma$, an ordering on $E(\Gamma)$ and an orientation on $E(\Gamma) \setminus E(\tree) = \{e_1, \ldots, e_g\}$, there is a filtration on $\C[\mathcal{X}(F_g, G)]$ with associated graded ring the coordinate ring of $P_{\Gamma}(G) = [\prod_{v \in V(\Gamma)} P_3(G)]/T^{E(\Gamma)}.$
\end{proposition}

\begin{proof}
We identify the ordered, oriented edges $e_1, \ldots, e_g$ with the components of $G^g$, where the orientation distinguishes the left and right hand side $G$ actions on each component. The filtration above then yields an associated graded algebra $\C[P_{2g}(G)/T^g]$. We split each edge $e_i$ in two: $f_{2i-1}, f_{2i},$ and build the trivalent tree $\tree'$ using the topology of the spanning tree $\tree,$ see Figure \ref{edgesplit}.  By Proposition \ref{treeweight}, this defines a filtration on $\C[P_{2g}(G)]$ with associated graded algebra $\C[P_{\tree'}(G)] = [\prod_{v \in V(\Gamma)} P_3(G)]/T^{E(\tree')}.$ By the $T^{2g}$ stability of this filtration and Propositions \ref{compositeprop} and \ref{Gassociatedgraded}, we may induce a filtration on $\C[\mathcal{X}(F_g, G)]$ with associated graded algebra the quotient $P_{\Gamma}(G) = P_{\tree'}(G)/T^g.$
\end{proof}

\begin{figure}[htbp]
\centering
\includegraphics[scale = 0.26]{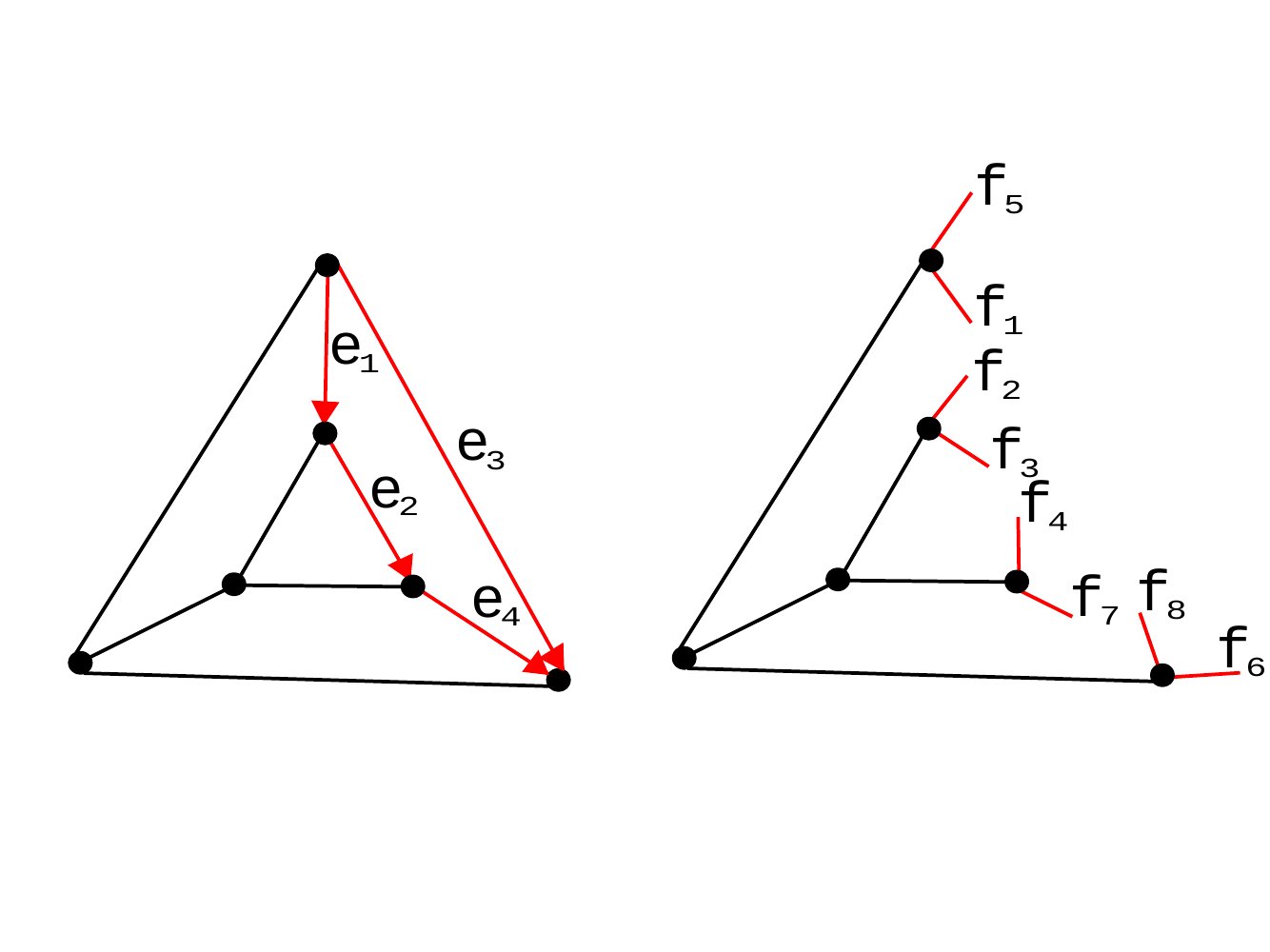}
\caption{Splitting the complement of a spanning tree.}
\label{edgesplit}
\end{figure}

\begin{example}\label{sl2example}
In the case $G = SL_2(\C),$ the spaces $W(r_1, r_2, r_3),$ $r_i \in \Z_{\geq 0} \subset \R_{\geq 0} = \Delta$ are multiplicity-free.  The dimension $W(r_1, r_2, r_3)$ is non-zero precisely when $r_1, r_2, r_3$ satisfy two conditions: $r_1 + r_2 +r_3 \in 2\Z$, and $|r_1 - r_2| \leq r_3 \leq r_1 + r_2.$  This second condition is equivalent to the $r_i$ being the sidelengths of a triangle.  The affine semigroup $BZ_3(SL_2(\C)) \subset \Z_{\geq 0}^3$ (see Section \ref{example} for an explanation of this notation) of the $(r_1, r_2, r_3)$ satisfying these conditions is generated by $(1, 1, 0), (1, 0, 1)$ and $(0, 1, 1)$.
In this case $T = \C^*$, and the action of $T^3 = (\C^*)^3$ defines the structure of an affine $3-$dimensional toric variety on $P_3(SL_2(\C))$, in particular this space is naturally identified with the affine space $\bigwedge^2(\C^3).$  By working through the $(\C^*)^{E(\Gamma)}$ quotient in Proposition \ref{characterstep1}, we identify $P_{\Gamma}(SL_2(\C))$ with an affine toric variety.  The associated affine semigroup $BZ_{\Gamma}(SL_2(\C))$ is the set of weightings $w: E(\Gamma) \to \Z_{\geq 0}$ such that for any $e, f, g \in E(\Gamma)$ sharing a common vertex $v \in V(\Gamma)$ we must have $w(e) + w(f) + w(g) \in 2\Z$ and $|w(e) - w(f) | \leq w(g) \leq w(e) + w(f).$
\end{example}

\subsection{Graph construction of character varieties}

We present an alternative construction of the variety $\mathcal{X}(F_g, G)$, which motivates the graph filtrations of Proposition \ref{characterstep1}. A variation on this construction has also been discovered by Florentino and Lawton in \cite{FL}.   We fix a trivalent graph $\Gamma$ with no leaves, and consider the forest $\hat{\Gamma}$ obtained by splitting each edge in $E(\Gamma).$

\begin{figure}[htbp]
\centering
\includegraphics[scale = 0.3]{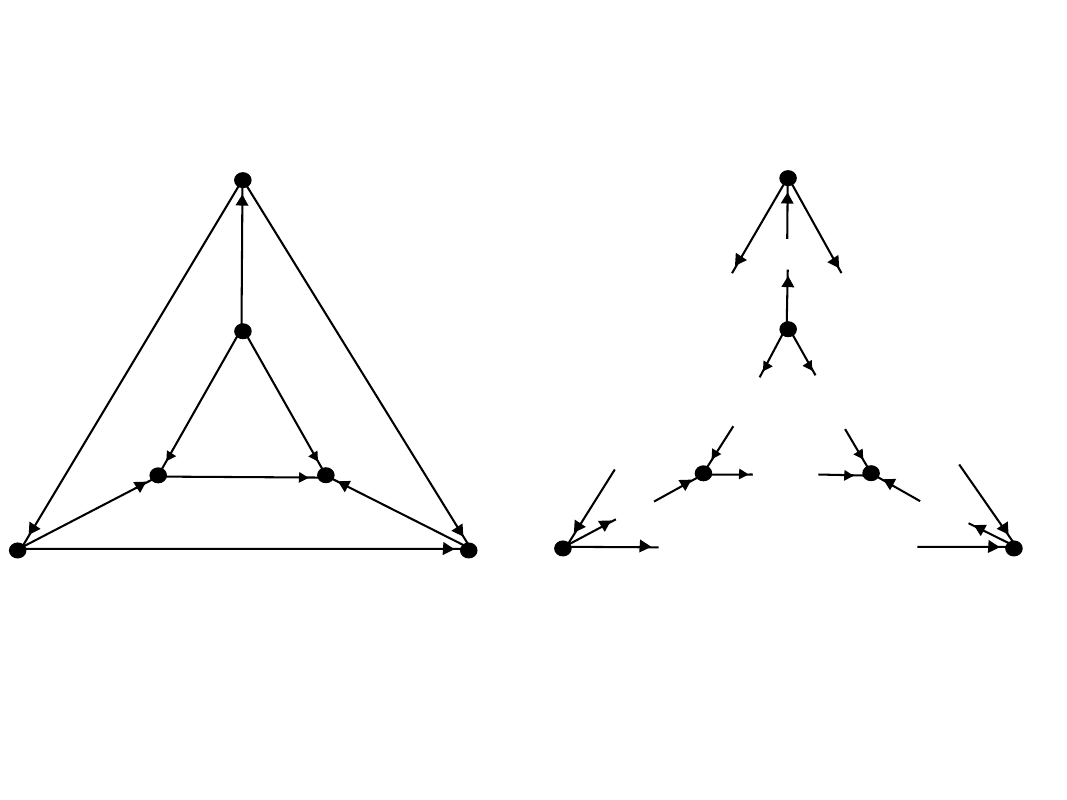}
\caption{The directed forest associated to a directed graph.}
\label{forestsplit}
\end{figure}

We associate a copy of the $GIT$ quotient $M_3(G) = G \backslash G^3$ to each connected component of $\hat{\Gamma},$  notice that this space has a residual right $G^3$ action.   The product $\prod_{v \in V(\hat{\Gamma})} M_3(G)$ then has an action by $E(\Gamma)$ copies of $G$, where the component corresponding to $e \in E(\Gamma)$ acts on the two components associated to the pair of edges in $\hat{\Gamma}$ which split $e$. We define $M_{\Gamma}(G)$ to be the $GIT$ quotient by this action. 

\begin{equation}
M_{\Gamma}(G) = [\prod_{v \in V(\Gamma)} M_3(G)]/G^{E(\Gamma)}\\
\end{equation}

\begin{proposition}\label{spantreeiso}
For each choice of a spanning tree $\tree \subset \Gamma$, an ordering on the edges $\vec{e} = E(\Gamma) \setminus E(\tree),$ and an orientation of each edge in $E(\Gamma),$ there is an isomorphism $\Phi_{\tree, \vec{e}}: M_{\Gamma}(G) \to \mathcal{X}(F_g, G).$
\end{proposition}

\begin{proof}
We split the edges $e_i \in \vec{e} \subset E(\Gamma)$ into pairs $e_{2i-1}, e_{2i}$, ordered using the orientation on $e_i,$ this gives a trivalent tree $\tree'.$  We construct $M_{\tree'}(G) = \prod_{v \in V(\Gamma)} M_3(G) / G^{E^o(\tree')},$ and note that $M_{\tree'}(G)/G^{\vec{e}} = M_{\Gamma}(G).$  We will prove that $M_{\tree'}(G) \cong G \backslash G^{2g}$ in Lemma \ref{treecontract}. We use the isomorphism $G_{2i-1}\times G_{2i}/G \cong G,$ $(g_{2i-1}, g_{2i}) \to g_{2i-1}g_{2i}^{-1}.$  This intertwines the left $G$ action on $G^{2g}$ with the adjoint action on $G^g.$
\end{proof}

Note that the definition of $M_{\Gamma}(G)$ is sensible for any graph, regardless of the valence of the vertices. We consider $M_{\tree}(G)$ for possibly non-trivalent trees in the following lemma. 

\begin{lemma}\label{treecontract}
For any tree $\tree$ with $n$ leaves, each orientation on the edges of $\tree$ gives an isomorphism to the left quotient $M_{\tree}(G) \cong G \backslash G^n.$ 
\end{lemma}

\begin{proof}
It suffices to treat the case where  $\tree$ has one internal edge $e,$ as this calculation can then be iterated to show the result by induction. Let $\partial(e) = \{u, w\}$, with the orientation pointing from $u$ to $v.$ We view $M_{\tree}(G)$ as $G^{V(u)} \times G^{V(w)}$ with a left action of $G \times G$  and a right action of $G$ on two components $G_{e, u} \subset G^{V(u)}, G_{e, w} \subset G^{V(w)}.$  We use the map $G^{V(u)} \times G^{V(w)} \to G^{V(u) + V(w) - 2}$ given by the following. 

\begin{equation}
(g_1, \ldots, g_{V(u)-1}, g_{e, u}) \times (g_{e, w}, g_{V(u) +2}, \ldots, g_{V(w)}) \to\\ 
\end{equation}

$$ (g_1, \ldots, g_{V(u)-1} ,g_{e, w}g_{e, u}^{-1},  g_{V(u) +2}, \ldots, g_{V(w)}) \to$$
$$ (g_1, \ldots, g_{V(u)-1},  g_{e, u}g_{e, w}^{-1}g_{V(u) +2}, \ldots g_{e, u}g_{e, w}^{-1}g_{V(w)})$$

This is a $G \times G \times G$, where the first component acts diagonally on the left of $ G^{V(u) + V(w) - 2}$, and the second and third components act trivially.  Quotienting everything by $G\times G\times G$ then yields the isomorphism.  
\end{proof}

We may also view $M_{\Gamma}(G)$ as the following quotient. 

\begin{equation}
\prod_{v \in V(\Gamma)} M_3(G) / G^{E(\Gamma)} = G^{V(\Gamma)} \backslash \prod_{v \in V(\Gamma)} G^3 / G^{E(\Gamma)}\\
\end{equation}

$$ = G^{V(\Gamma)} \backslash \prod_{e \in E(\Gamma)} [(G\times G)/G] = G^{V(\Gamma)} \backslash G^{E(\Gamma)}.$$

\noindent
We can now recover Proposition \ref{characterstep1} by replacing the rightmost term with the horospherical degeneration $G^{V(\Gamma)} \backslash \prod_{e \in E(\Gamma)} [G/U_- \times U_+ \backslash G]/T = P_{\Gamma}(G).$

\section{Valuations from the dual canonical basis}\label{step2}

In Propositions \ref{treeweight} and  \ref{characterstep1}, each graph $\Gamma$ (respectfully tree $\tree$) with compatible information defines the following direct sum decompositions on the coordinate rings $\C[\mathcal{X}(F_g, G)]$ and $\C[P_n(G)]$. Here $W(\lambda, \eta, \mu)$ denotes the invariant vectors in $V(\lambda) \otimes V(\eta) \otimes V(\mu)$. 

\begin{equation}
\C[\mathcal{X}(F_g, G)] = \bigoplus_{\lambda: E(\Gamma) \to \Delta} \bigotimes_{v \in V(\Gamma)} [W(\lambda(v, i), \lambda(v, j), \lambda(v, k))]\\
\end{equation}

\begin{equation}
\C[P_n(G)] = \bigoplus_{\lambda: E(\tree) \to \Delta} \bigotimes_{v \in V(\tree)} [W(\lambda(v, i), \lambda(v, j), \lambda(v, k))]\\
\end{equation}

\noindent
Here $(v, i)$ denotes a vertex $v$ with incident edge $i$. In this section we show how to obtain finer combinatorial pictures of $\C[\mathcal{X}(F_g, G)]$ and $\C[P_n(G)]$ by structuring the intertwiner spaces $W(\lambda, \eta, \mu)$ using the dual canonical basis.

\subsection{String parameters and polytopes for tensor product multiplicities}

We recall the construction of polyhedral cones $C_{\bold{i}}(3)$ which control tensor product multiplicites for a reductive group $G.$ We take $G$ to be semisimple, but we will later remove this restriction.  In \cite{Lu},  Lusztig constructs a basis $\mathbb{B}$ of the subalgebra $\mathcal{U}_q(\mathfrak{u}_+)$ of the quantized universal enveloping algebra $\mathcal{U}_q(\mathfrak{g}).$ Specialization at $q = 1$ yields the canonical basis for each irreducible representation $V(\lambda) \subset \mathcal{U}(u_+)$.  The dual pairing between $\mathcal{U}(\mathfrak{u}_+)$ and $\C[U_+]$ induces a dual basis  $B(\lambda^*) \subset V(\lambda^*) \subset \C[U_+]$.  A basis $B \subset \C[U_- \backslash G]$ can then be constructed by taking the union $B = \coprod_{\lambda} B(\lambda) \times \{\lambda\}$. 

 We fix a reduced decomposition $\bold{i} \in R(w_0)$.  The entries of $\bold{i}$ correspond to to raising operators $e_{i_1}, \ldots, e_{i_N} \in \mathfrak{u}_+,$ these act as nilpotent derivations on $\C[U_+]$ from the differentiation of the left and the right actions of $U_+$ on itself.  We use the notation $e \circ_{\ell} f$ and $e \circ_r f$ to distinguish between the left and right actions of such an operator, and repeated applications are denoted with a superscript $e \circ \ldots \circ e = e^n.$  For the following see \cite{C} and \cite{K}. 

\begin{definition}
  For $f \in \C[U_+]$, the string parameters $w_{\bold{i}}(f) = (t_1, \ldots, t_N) \in \Z^N$ are defined inductively, starting with $t_1 = min\{ t | e_{i_1}^{t+1} \circ_{\ell} f = 0\}$.  The $k-$th entry is then defined as $t_k = min\{ t | e_{i_k}^{t+1} \circ_{\ell} e_{i_{k-1}}^{t_{k-1}}\circ_{\ell} \ldots \circ_{\ell}e_{i_1}^{t_1} \circ_{\ell} f = 0\}.$  This defines the string parametrization $w_{\bold{i}}: \C[U_+]\setminus \{0\} \to \Z^N$ associated to $\bold{i} \in R(w_0).$
\end{definition} 

\begin{lemma}
The function $w_{\bold{i}}$ is a valuation on $\C[U_+]$ when the string parameters are lex ordered first to last. 
\end{lemma}

\begin{proof}
It is straightforward to show $w_{\bold{i}}(C) = (0, \ldots, 0)$ for any $C \in \C \setminus \{0\}.$ 
For $f, g \in \C[U_+]$ let $e_{i_k}$ be the first raising operator for which the string parameters $w_{\bold{i}}(f), w_{\bold{i}}(g)$ differ, with $t_k(f) > t_k(g)$. Then by definition, $w_{\bold{i}}(f + g) = w_{\bold{i}}(f).$ If no such $k$ exists, then $w_{\bold{i}}(f + g) \leq w_{\bold{i}}(f) = w_{\bold{i}}(g)$.  Applying $e_{i_1}^M$ to $fg$ yields the following. 

\begin{equation}
e_{i_1}^M(fg) = \sum_{p + q = M} \binom{M}{p} e_{i_1}^p(f)e_{i_1}^q(g)
\end{equation}

\noindent
If $M > t_1(f) + t_1(g),$ all terms in this sum must vanish. If $M \leq t_1(f) + t_1(g),$ then the fact that $\C[U_+]$ is a domain implies that this sum is non-zero, as this is the case for $M = t_1(f) + t_1(g),$ where the sum is multiple of $e_{i_1}^{t_1(f)}(f)e_{i_1}^{t_1(g)}(g).$ We may repeat this calculation with $e_{i_2}$ on this term.  By induction this yields $w_{\bold{i}}(fg) = w_{\bold{i}}(f) + w_{\bold{i}}(g).$
\end{proof}

We obtain a valuation $v_{\bold{i}}$ on the coordinate ring $\C[ U_- \backslash G] \subset \C[T \times U_+]$, with image in $\Z^N \times \Delta$ by breaking ties with the $\bold{<}$ ordering on $\Delta.$  Berenstein and Zelevinsky, \cite{BZ1} and Alexeev and Brion, \cite{AB} give inequalities for the image of this valuation using devices derived from the fundamental representations $V(\omega_i)$ of the Lie algebra $\breve{\mathfrak{g}}$ of the Langlands dual group $\breve{G}$, called $\bold{i}-$trails.  An $\bold{i}-$trail from a weight $\gamma$ to a weight $\eta$ in the weight polytope of a representation $V$ is a sequence of weights $(\gamma, \gamma_1, \ldots, \gamma_{\ell-1}, \eta),$ such that consecutive differences of weights are integer multiples of simple roots from $\bold{i}$, $\gamma_i - \gamma_{i+1} = c_k \alpha_{i_k},$ and the application of the raising operators $e_{i_1}^{c_1} \circ \ldots \circ e_{i_{\ell}}^{c_{\ell}}: V_{\eta} \to V_{\gamma}$ is non-zero.  For any $\bold{i}-$trail $\pi,$ Berenstein and Zelevinsky define  $d_k(\pi) = \frac{1}{2}(\gamma_{k-1} + \gamma_k)(H_{i_k}).$   For the next proposition see \cite{AB}, \cite{K}, and \cite{BZ1}. 

\begin{proposition}\label{flagimage}
The valuation $v_{\bold{i}}$ defines a bijection between $B \subset \C[U_- \backslash G]$ and the integral points in a convex polyhedral cone $C_{\bold{i}} \subset \Z^N \times \Delta$ defined by the following inequalities: 

\begin{enumerate}
\item $\sum_k d_k(\pi) t_k \geq 0$ for any $\bold{i}-$trail $\omega_i \to w_0s_i\omega_i$ in $V(\omega_i),$ for all fundamental weights $\omega_j$ of $\breve{\mathfrak{g}}$.\\   
\item $t_k \leq \lambda(H_{\alpha_{i_k}}) -\sum_{\ell = k+1}^N a_{i_{\ell}, i_k} t_{\ell}$ for $k = 1, \ldots, N.$\\
\end{enumerate}

\noindent
In particular $v_{\bold{i}}(B) = v_{\bold{i}}(\C[U_- \backslash G])$, $v_{\bold{i}}$ is a maximal rank valuation with image equal to a normal
affine semigroup, and the Newton-Okounkov body of $v_{\bold{i}}$ is $C_{\bold{i}}.$  

\end{proposition}

\begin{remark}
Any set $e_1, \ldots, e_m$ of $k-$linear nilpotent operators on a $k-$algebra $A$ defines a valuation in this way.  It would be very useful to know general sufficient conditions for the value semigroup $v_{\vec{e}}(A) \subset \Z_{\geq 0}^m$ to be finitely generated.
\end{remark}

For $b \in B$, with $v_{\bold{i}}(b) = ( \vec{t}, \lambda) \in C_{\bold{i}},$ the tuple $\vec{t} \in \Z^N$ is called the string parameter of $b$ associated to $\bold{i}.$   As constructed, the basis $B$ is composed of $T\times T$-weight vectors, with weights $(\sum t_i\alpha_i - \lambda, \lambda).$  In particular, $v_{\bold{i}}$ is a $T\times T$-stable valuation.  The filtration $F^{\bold{i}}$ corresponding to this valuation is given by $T\times T$-stable subspaces $F^{\bold{i}}_{\leq (\vec{t}, \lambda)} \subset \C[U_- \backslash G]$, which has a basis of those $b \in B$ with $v_{\bold{i}}(b) \leq (\vec{t}, \lambda).$

Now we consider the spaces $V_{\beta, \gamma}(\lambda) \subset V(\lambda),$ defined as the collection of those vectors of weight $\gamma$ which are annhilated by the raising operators $e_i^{\beta(H_{\alpha_i}) + 1}$ (acting on the right).  The basis $B$ has the ``good basis'' property (see \cite{Mat}), this implies that $B_{\beta, \gamma}(\lambda) = B \cap V_{\beta, \gamma}(\lambda)$ is a basis for the space $V_{\beta, \gamma}(\lambda)$.  In the case $\beta = \eta$, and $\gamma = \mu^* - \eta$ this space is classically known to be isomorphic to the invariant space $W(\mu, \lambda, \eta)$ (see \cite{Zh}, we will also provide a proof in the next subsection).    Berenstein and Zelevinksy characterize the string parameters $\vec{t}$ corresponding to the $b \in B_{\eta, \mu^* - \eta}(\lambda)$ as follows.

\begin{theorem}[Berenstein, Zelevinksy, \cite{BZ1}]\label{tenspolytope}
The decomposition $\bold{i}$ gives a labelling of $B_{\eta, \mu^* - \eta}(\lambda)$ by the points in $\Z_{\geq 0}^N$ such that the following hold: 

\begin{enumerate}
\item $\sum_k d_k(\pi)t_k \geq 0$ for any $\bold{i}-$trail
from $\omega_j$ to $w_0 s_j\omega_j$ in $V(\omega_j),$ for all fundamental weights $\omega_j$ of $\breve{\mathfrak{g}}$.\\

\item $-\sum_k t_k \alpha_k + \lambda  + \eta = \mu^*$\\

\item $\sum_k d_k(\pi) t_k \geq \eta(H_{\alpha_j})$  for any $\bold{i}-$trail
from $s_j\omega_j$ to $w_0\omega_j$ in $V(\omega_j),$  for all fundamental weights $\omega_j$ of $\breve{\mathfrak{g}}$.\\

\item $t_k + \sum_{j > k} a_{i_k, i_j} t_j \geq \lambda(H_{\alpha_{i_k}})$\\
\end{enumerate}

These are the integral points in a polytope $C_{\bold{i}}(\mu, \lambda, \eta).$
\end{theorem}

The first and last conditions say that $(\vec{t}, \lambda)$ is a member
of $C(\bold{i})$ in the fiber over the weight $\lambda,$ the second condition
says that the basis members lie in the weight $\mu^* - \eta$ subspace of $V(\lambda),$
and the third condition says that the appropriate raising operators $e_i^{\eta(H_{\alpha_i})+1}$ annihilate. We realize these polytopes as slices of the following polyhedral cone.

\begin{definition}
For a string parameterization $\bold{i},$ the cone $C_{\bold{i}}(3)$ is defined by the following inequalities on $(\vec{t}, \lambda, \eta) \in C(\bold{i}) \times \Delta \subset  \Z_{\geq 0}^N \times \Delta \times \Delta:$

\begin{enumerate}
\item $\sum_k d_k(\pi)t_k \geq 0$ for any $\bold{i}-$trail
from $\omega_j$ to $w_0 s_j\omega_j$ in $V(\omega_j),$  for all fundamental weights $\omega_j$ of $\breve{\mathfrak{g}}$.\\

\item $-\sum_k t_k \alpha_k + \lambda  + \eta \in \Delta$\\

\item $\sum_k d_k(\pi) t_k \geq \eta(H_{\alpha_j})$  for any $\bold{i}-$trail
from $s_j\omega_j$ to $w_0\omega_j$ in $V(\omega_j),$   for all fundamental weights $\omega_j$ of $\breve{\mathfrak{g}}$.\\

\item $t_k + \sum_{j > k} a_{i_k, i_j} t_j \geq \lambda(H_{\alpha_{i_j}})$\\
\end{enumerate}

\end{definition}

We let $\pi_1( \vec{t}, \lambda,  \eta) = (-\sum_k t_k \alpha_k + \lambda  + \eta)^*,$ $\pi_2( \vec{t}, \lambda,  \eta) = \lambda$, and $\pi_3( \vec{t}, \lambda, \eta) = \eta.$  By construction $C_{\bold{i}}(\mu, \lambda, \eta)$ is the fiber of these maps over $(\mu, \lambda, \eta).$

\subsection{The tensor product ring map}

The $T\times T-$stable subspace $\bigoplus V_{\eta, \mu^*-\eta}(\lambda)t^{\eta}$ $\subset \C[U_- \backslash G \times T]$ inherits a basis $B_3$ from $B \times \Delta \subset \C[U_- \backslash G \times T]$. In this subsection we show that this space is an algebra, isomorphic to $\C[P_3(G)]$.   As a result we will obtain both a basis $B_3 \subset \C[P_3(G)]$ and a $T^3-$invariant valuation $v_{\bold{i},3}: \C[P_3(G)] \to C_{\bold{i}}(3)$ with $v_{\bold{i}, 3}(B_3)$ equal to the integer points in $C_{\bold{i}}(3).$

We use the isomorphism $P_3(G) \cong (U_- \backslash G \times U_- \backslash G \times G/U_+)/G \cong (U_- \backslash G \times U_- \backslash G)/U_+.$ 
The residual action of $T^3$ on $P_3(G)$ intertwines under this isomorphism with a torus $T^3 \cong T_1 \times T_2 \times T_3$ action on $(U_- \backslash G \times U_- \backslash G)/U_+$ defined as follows: the tori $T_2$ and $T_3$ act on the left hand sides of the components $U_- \backslash G \times U_- \backslash G$, and the torus $T_1$ acts diagonally on the right through the automorphism of $T$ corresponding to the dual map $-w_0: \mathfrak{t}^* \to \mathfrak{t}^*$ on characters.   We make use of the following commutative diagram of affine varieties.

$$
\begin{CD}\label{birat}
[T \times U_+ \times T \times U_+]/U_+ @<\pi<< T \times U_+ \times T \\
@VVV @VVV\\
[U_- \backslash G \times U_- \backslash G]/U_+ @<<< U_- \backslash G \times T\\
\end{CD}
$$

The top row is the map $\pi: (s, u, t) \to (s, u, t, Id)$, this is an isomorphism with inverse $(s, x, t, y) \to (xy^{-1}, s, t)$. This map intertwines the left $U_+ \times U_+$ action on $(U_+ \times U_+)/ U_+$ with the left and right actions of $U_+$ on itself. The bottom row is defined similarly, and the vertical arrows are given by the map $(s, u) \to su \in TU_+ \subset U_- \backslash G.$   

We identify $T \times U$ with the Borel a subgroup $B \subset G$, which is isomorphic as a group to $T \ltimes U$. The $T_1$ action is the defined as the right hand side action of the torus on $B$.  Similarly the $T_2$ and $T_3$ actions are defined as the left hand actions on the first and second components $T \times U_+,$ respectively. These actions make all of the maps in the above diagram $T_1 \times T_2 \times T_3$-equivariant.  Now we recall that the irreducible representation $V(\lambda)$ has the following description as a subspace of $\C[U_+]$ (\cite{Mat}, \cite{Zh}). 

\begin{equation}
V(\lambda) = \{f \in \C[U_+] | e_i^{\lambda(H_{\alpha_i}) +1} \circ_{\ell} f = 0\}\\ 
\end{equation}

\noindent
Here $1 \in \C[U_+]$ is identified with the highest weight vector $b_{\lambda} \in V(\lambda).$ We let $V_{\eta}(\lambda) \subset V(\lambda)$ denote the space of functions $f$ which satisfy $e_i^{\eta(H_{\alpha_i}) + 1} \circ_r f = 0.$

\begin{lemma}
The following diagram commutes, and the top row is an isomorphism of vector spaces. 

$$
\begin{CD}
(V(\lambda) \otimes V(\eta))^{U_-} @>>> V_{\eta}(\lambda)\\
@VVV @VVV\\
\C[(T \times U_+ \times T \times U_+)/U_+ ] @>\pi^*>> \C[T \times U_+ \times T]\\
\end{CD}
$$

\end{lemma}

\begin{proof}
We take a function $f \in V(\lambda) \otimes V(\eta) \subset \C[(T \times U_+ \times T \times U_+)/U_+]$ and analyze the pullback $\pi^*(f).$  The function $f$ satisfies the equations $e_i^{\lambda(H_{\alpha_i}) + 1} \circ_{\ell} f = 0$ in the first $U_+$ component and $e_i^{\eta(H_{\alpha_i}) + 1} \circ_{\ell} f = 0$ in the second.  By definition of $\pi$, these equations are satisfied if and only if $e_i^{\lambda(H_{\alpha_i}) + 1} \circ_{\ell} \pi^*(f) = 0$, and $e_i^{\eta(H_{\alpha_i}) + 1} \circ_r \pi^*(f) = 0.$    
\end{proof}

Now we consider what happens when $f \in (V(\lambda)\otimes V(\eta))^{U_+}$ has weight $\mu^*$, this is the case when $f$ represents an intertwiner $V(\mu) \to V(\lambda) \otimes V(\eta).$   In the coordinate ring $\C[G] = \bigoplus_{\lambda \in \Delta} V(\lambda) \otimes V(\lambda^*)$, specialization at $Id$ is contraction of $V(\lambda)$ against the dual $V(\lambda^*).$  The coordinate ring $\C[U_- \backslash G] \subset \C[G]$ is the subalgebra of spaces $\C v_{-\eta} \otimes V(\eta),$ it therefore follows that $\pi^*(f) \in \C[U_+]$ is the coefficient of the $v_{\eta}$ component of $f$.  Since $f$ was chosen to have weight $\mu^*$, $\pi^*(f) \in V_{\eta}(\lambda)$ must be a $\mu^* - \eta$ weight vector.  

\begin{lemma}
The space $W(\mu, \lambda, \eta) \cong (V(\lambda) \otimes V(\eta))^{U_-}_{\mu^*}$ is isomorphic to $V_{\eta, \mu^* - \eta}(\lambda).$ 
\end{lemma}

\begin{proof}
We have already established a $1-1$ map $\pi^*: (V(\lambda) \otimes V(\eta))^{U_+}_{\mu^*} \to V_{\eta, \mu^* - \eta}(\lambda)$.  To show that this map is also onto, we observe that $(V(\lambda) \otimes V(\eta))^{U_+}$ is a direct sum of its dominant weight spaces, each of which maps to a distinct $V_{\eta, \mu^* - \eta}(\lambda) \subset V_{\eta}(\lambda),$ and that $(V(\lambda) \otimes V(\eta))^{U_+} \cong V_{\eta}(\lambda).$
\end{proof}

The torus $T_1 \times T_2 \times T_3$ acts on the space $V_{\eta, \mu^* - \eta}(\lambda)$ with character $((\mu^* - \eta + \eta)^*, \lambda, \eta) =(\mu, \lambda, \eta).$ We have now established a $T^3-$stable map of algebras, identifying $\C[P_3(G)]$ with the subspace $\bigoplus_{\mu, \lambda, \eta} V_{\eta, \mu^* - \eta}(\lambda)t^{\eta} \subset \C[U_- \backslash G \times T].$

\begin{theorem}\label{3tensor}
For each $\bold{i}$ there is a $T^3$ stable valuation $v_{\bold{i}, 3}$ on $\C[P_3(G)]$, with associated graded ring equal to the affine semigroup algebra $\C[C_{\bold{i}}(3)].$ The torus $T_1 \times T_2 \times T_3$ acts on $\C[C_{\bold{i}}(3)]$ with characters $\pi_1, \pi_2, \pi_3: C_{\bold{i}}(3) \to \Delta$. Furthermore, $v_{\bold{i}, 3}(B_3)$ coincides with the integer points in $C_{\bold{i}}(3).$
\end{theorem}

\begin{proof}
The valuation $v_{\bold{i}, 3}$ is constructed from $\bold{<}$ on $\Delta$
and $v_{\bold{i}}$ on $\C[U_- \backslash G].$  It is invariant with respect to $T_1 \times T_2 \times T_3,$ because $v_{\bold{i}}, \bold{<}: \C[U_- \backslash G \times T \to C_{\bold{i}}$ is $T^4$ invariant. By construction the character with respect to the torus action on $\C[C_{\bold{i}}(3)]$ corresponds to the maps $\pi_1, \pi_2, \pi_3$.  The image $v_{\bold{i}, 3}(\C[P_3(G)])$ is then the image of those $b\otimes t^{\eta} \in \C[U_-\backslash G \times T]$ which lie in $\bigoplus_{\mu, \lambda, \eta} V_{\eta, \mu^* - \eta}(\lambda)t^{\eta}$ under $v_{\bold{i}}, \bold{<}$.  This coincides with the integer points in $C_{\bold{i}}(3)$ by construction. 
\end{proof}

This exposition has been for the semisimple case, but as in \cite{AB}, everything can be generalized readily to the reductive case. The weights that define a non-zero $W(\mu, \lambda, \eta)$ are of the form $\mu' + \tau_1,$ $\eta' + \tau_2$ and $\lambda' + \tau_3$ where $\tau_i$ are characters of the center $Z(G)$ with $\tau_1 + \tau_2 + \tau_3 = 0,$ and $\mu', \eta', \lambda'$ are dominant weights of the semisimple part of $G.$  The subspace $V_{\eta,  \mu - \eta}(\lambda)$ is the same as the subspace $V_{\eta', \mu' - \beta' + (\tau_1 - \tau_2)}(\lambda' + \tau_3) = V_{\eta', \mu' -\eta' + \tau_3}(\lambda' + \tau_3) = V_{\eta', \mu' - \eta'}(\lambda')\otimes \C\tau_3.$  So this space inherits the subset of the dual canonical basis of the semi-simple part of $G$ coming from $V_{\eta', \mu' - \beta'}(\lambda')$ tensored with the character $\tau_3.$

\section{Proof of Theorems \ref{character} and \ref{configuration} }\label{laststep}

We construct filtrations on $\C[\mathcal{X}(F_g, G)]$ and $\C[P_n(G)]$ with toric associated graded algebras $\C[C_{\bold{i}}(\Gamma)], \C[C_{\bold{i}}(\tree, \vec{\lambda})].$    Choose $\Gamma$, with a total ordering on $E(\Gamma)$, a total ordering on $V(\Gamma)$, and an assignment $\bold{i}: V(\Gamma) \to R(w_0).$ 

The efforts of Section \ref{step1} give a filtration on $\C[\mathcal{X}(F_g, G)]$ by $(\Delta, <)^{E(\Gamma)}$ with associated graded algebra $\C[P_{\Gamma}(G)].$  In Section \ref{step2} we construct a $T^3$ invariant valuation on  $\C[P_3(G)]$ with associated graded algebra $\C[C_{\bold{i}}(3)].$ We use the ordering on $V(\Gamma)$ and the assignment $\bold{i}: V(\Gamma) \to R(w_0)$ to define a full rank, $T^{E(\Gamma)}-$stable valuation on $\bigotimes_{v \in V(\Gamma)} \C[P_3(G)]$ with associated graded algebra $\bigotimes_{v \in V(\Gamma)} \C[C_{\bold{i}(v)}(3)]$.  This valuation induces a valuation on the sub-algebra $\C[P_{\Gamma}(G)] \subset \bigotimes_{v \in V(\Gamma)} \C[P_3(G)]$  of $T^{E(\Gamma)}$ invariants, and by Theorem \ref{3tensor} and Proposition \ref{Gassociatedgraded} the associated graded algebra is the semigroup algebra of the following polyhedral cone. 

\begin{definition}
We let $\pi_{v, i}$ be the projection map on $C_{\bold{i}(v)}(3)$ defined by the edge $i$
incident on the vertex $v \in V(\Gamma).$ 

We define $C_{\bold{i}}(\Gamma)$ to be the toric fiber product cone in $\prod_{v \in V(\Gamma)} C_{\bold{i}(v)}(3)$ defined by the conditions $\pi_{v, i} = \pi_{u, i}^*$ for all 
edges $i$ with endpoints $u, v$. 
\end{definition} 

 Theorem \ref{character} follows from Proposition \ref{compositeprop}. The same program can be carried out on the algebra $\C[P_n(G)]$, giving a $T^n$-stable valuation with associated graded algebra $\C[C_{\bold{i}}(\tree)].$  Theorem \ref{configuration} then follows by specializing the weights at the leaves of $\tree$ to $\vec{\lambda},$ using Theorem \ref{3tensor}.

\begin{proposition}\label{configbasis}
For every choice of a trivalent tree $\tree$ with $n$ ordered leaves, and an assignment $\bold{i}: V(\tree) \to R(w_0)$ we have:

\begin{enumerate}
\item a basis $B(\tree, m\vec{\lambda}) \subset H^0(P_{\vec{\lambda}^*}(G), \mathcal{L}(m\vec{\lambda}^*)) =$  $(V(\lambda_1) \otimes \ldots \otimes V(\lambda_n))^G$,\\
\item a labelling $v_{\tree, \bold{i}}: B(\tree, m\vec{\lambda}) \to C_{\bold{i}}(\tree, m\vec{\lambda}) \subset (\Delta \times \Z_{\geq 0}^N \times \Delta)^{V(\tree)}$.\\
\end{enumerate}

\end{proposition}

We conclude by showing that the image of the valuations we have constructed
coincide with all of the lattice points in their corresponding Newton-Okounkov bodies.  This implies that the corresponding affine semigroup algebras are normal.

\begin{proposition}\label{normalsemigroup}
The integer points in $C_{\bold{i}}(\Gamma)$ and $C_{\bold{i}}(\tree)$ are in bijection with the images of the induced valuations $v_{\bold{i}, \Gamma}(\C[\mathcal{X}(F_g, G)])$ and $v_{\bold{i}, \tree}(\C[P_n(G)]),$ respectively.
\end{proposition}

\begin{proof}
Each basis member of $\C[\mathcal{X}(F_g, G)]$ gives an element of $C_{\bold{i}}(\Gamma)$ by the constructions in Sections \ref{step1}, \ref{step2}. If $\vec{t} \in C_{\bold{i}}(\Gamma)$ is an integer point, it is likewise an integer point in $(\Delta \times \Z_{\geq 0}^N \times \Delta)^{V(\tree)}$, and therefore corresponds to a product $\otimes_{v \in V(\Gamma)} b_{v}$ of dual canonical basis elements which defines an regular function in $\C[\mathcal{X}(F_g, G)]$ with value equal to $\vec{t}$ by construction.  The same argument applies in the case $\C[P_n(G)].$
\end{proof}

\section{Examples}\label{example}

We describe the cones $C_{\bold{i}}(3)$ for $G= SL_m(\C)$ and all rank $2$ simple, simply-connected groups.  For $G = SL_m(\C)$ we describe particular instances of $B_{\bold{i}}(\Gamma), B_{\bold{i}}(\tree).$  The inequalities we present are culled from \cite{BZ2}, the treatment by Littelman \cite{Li}, and Theorem \ref{tenspolytope}. 

\subsection{Type A}

For $G = SL_m(\C)$, we take $\bold{i}$ to be the ``good'' decomposition $w_0 = s_1(s_2s_1)(s_3s_2s_1) \ldots (s_{m-1}\ldots s_1)$, (see \cite{Li}). The polyhedron $C_{\bold{i}}(3)$ is then the cone $BZ_3(SL_m(\C))$ of Berenstein-Zelevinsky triangles \cite{BZ3}.  

\begin{definition}
For this definition we refer to Figure \ref{triangle}.  A BZ triangle $T \in BZ_3(SL_m(\C))$
is an assignment of non-negative integers to vertices of the version of the diagram in Figure \ref{triangle} with $2(m-1)$ vertices on a side.  If $v, w$ are a pair of vertices which are across a hexagon from a pair $u, y$, then $T(v) + T(w) = T(u) + T(y).$ 

We let $a_1, \ldots, a_{2m-2}$  $= b_1, \ldots, b_{2m-2} =$  $c_1, \ldots, c_{2m-2} = a_1$ label the vertices clockwise around the boundary of the diagram.  This lets us define the following three projection maps $\pi_1, \pi_2, \pi_3: BZ_3(SL_m(\C)) \to \Delta_{SL_m(\C)}.$ 

\begin{equation}
\pi_1(T) = (a_1 + a_2, \ldots, a_{2m-3} + a_{2m-2})\\
\end{equation}

$$ 
\pi_2(T) = (b_1 + b_2, \ldots, b_{2m-3} + b_{2m-2})
$$

$$
\pi_3(T) = (c_1 + c_2, \ldots, c_{2m-3} + c_{2m-2})
$$ 

\end{definition}

\begin{figure}[htbp]
\centering
\includegraphics[scale = 0.26]{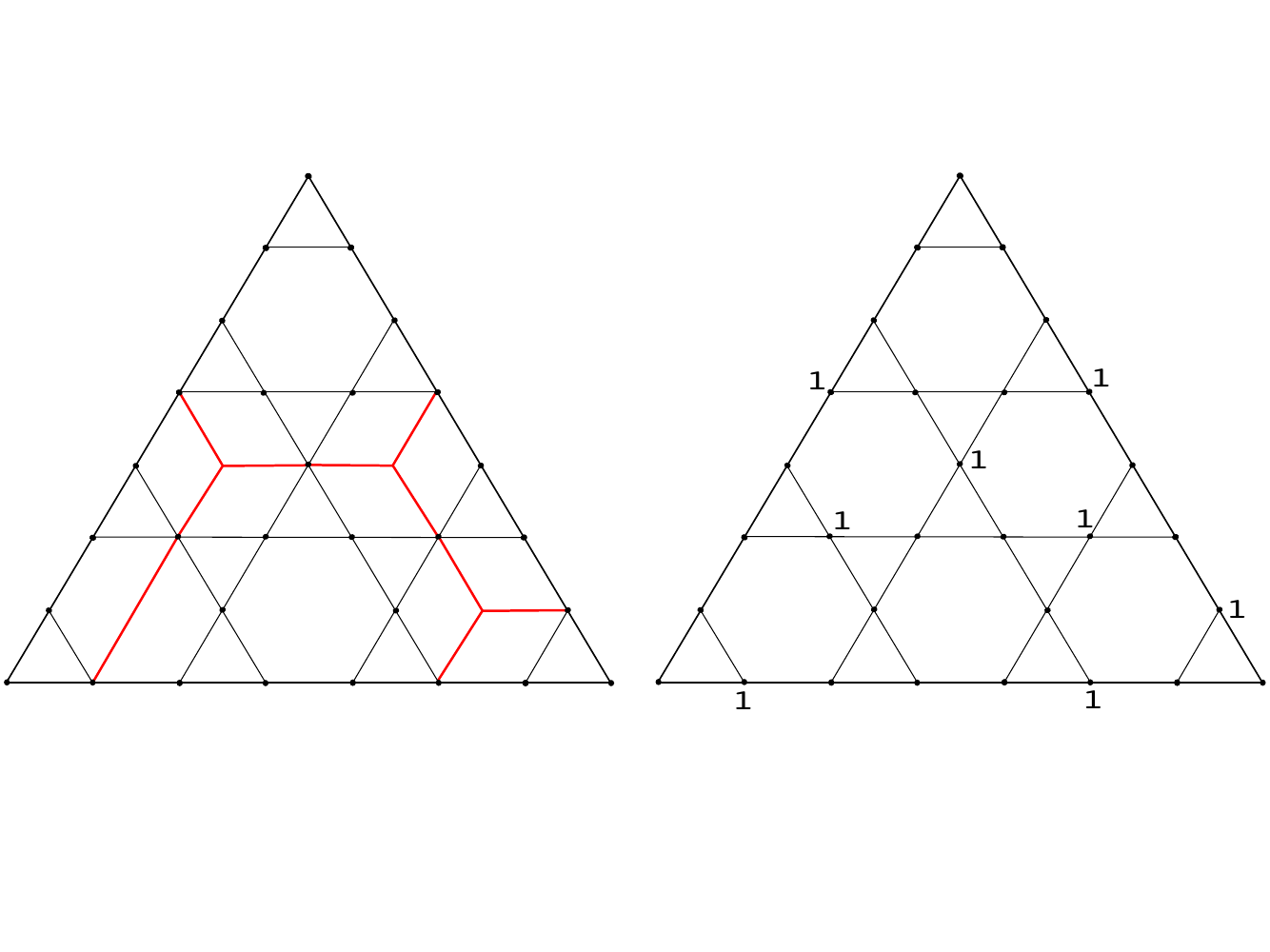}
\caption{Left: A honeycomb weighting. Right: Dual weighting by non-negative integers.}
\label{triangle}
\end{figure}

We can associate a dual graph to each $T \in BZ_3(SL_m(\C))$ by replacing each entry of weight $a$ with an edge to the center of its adjacent hexagon weighted $a$. The resulting graphs are also called honeycombs, and have been studied by a number of authors, see e.g. \cite{KTW04}, \cite{GP}.

\begin{figure}[htbp]
\centering
\includegraphics[scale = 0.27]{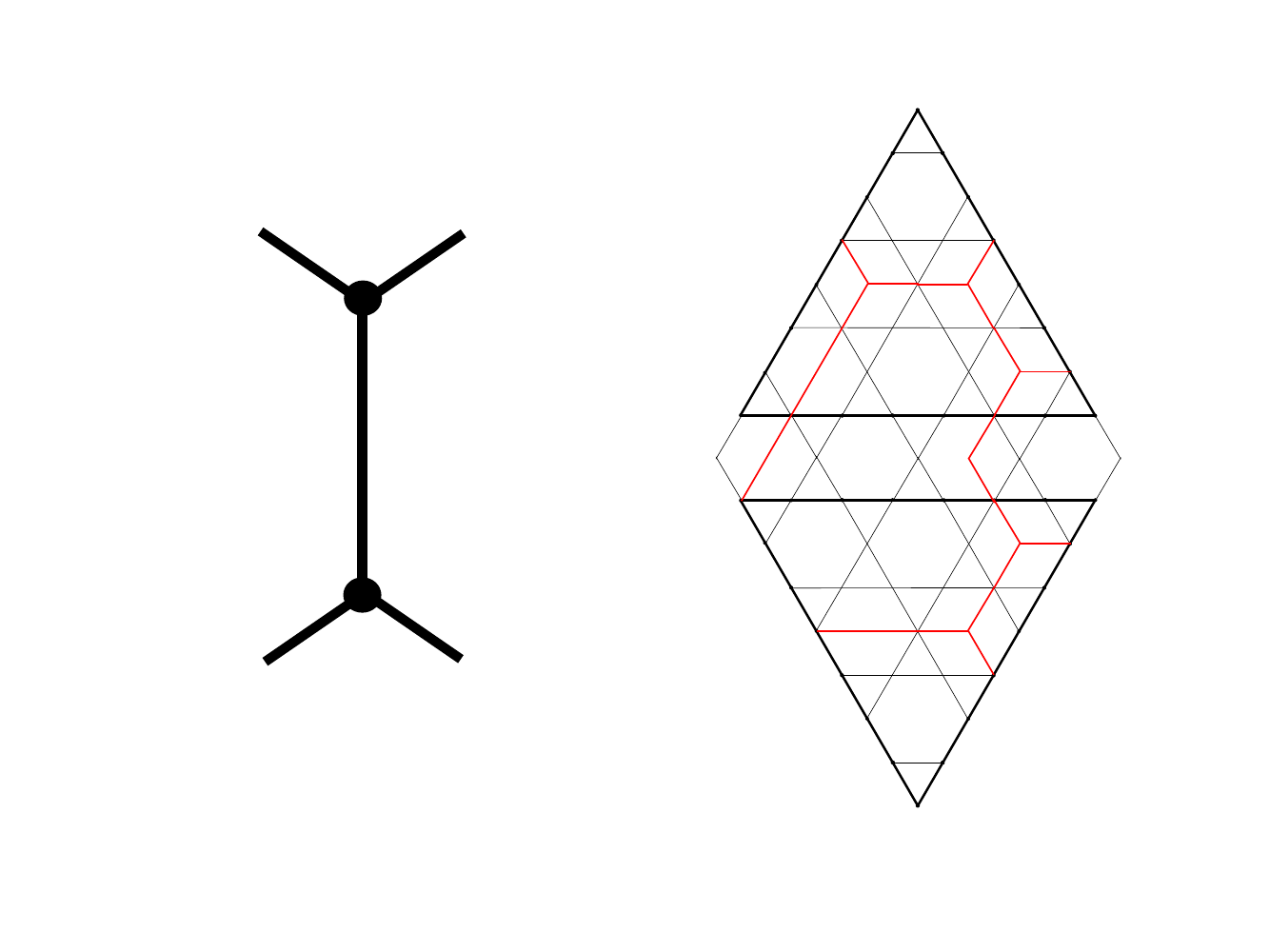}
\caption{A BZ quilt with dual $4-$tree.}
\label{quilt}
\end{figure}

The cone of $\Gamma-$BZ quilts $BZ_{\Gamma}(SL_m(\C)) \subset \prod_{v \in V(\Gamma)}BZ_3(SL_m(\C))$ is then defined to be those tuples $(T_v)$ with $\pi_{i, v}(T_v) = \pi_{i, u}(T_u)^*$ when an edge $i$ joins $u$ and $v,$  see Figures \ref{quilt} and \ref{quilt2}.

\begin{figure}[htbp]
\centering
\includegraphics[scale = 0.27]{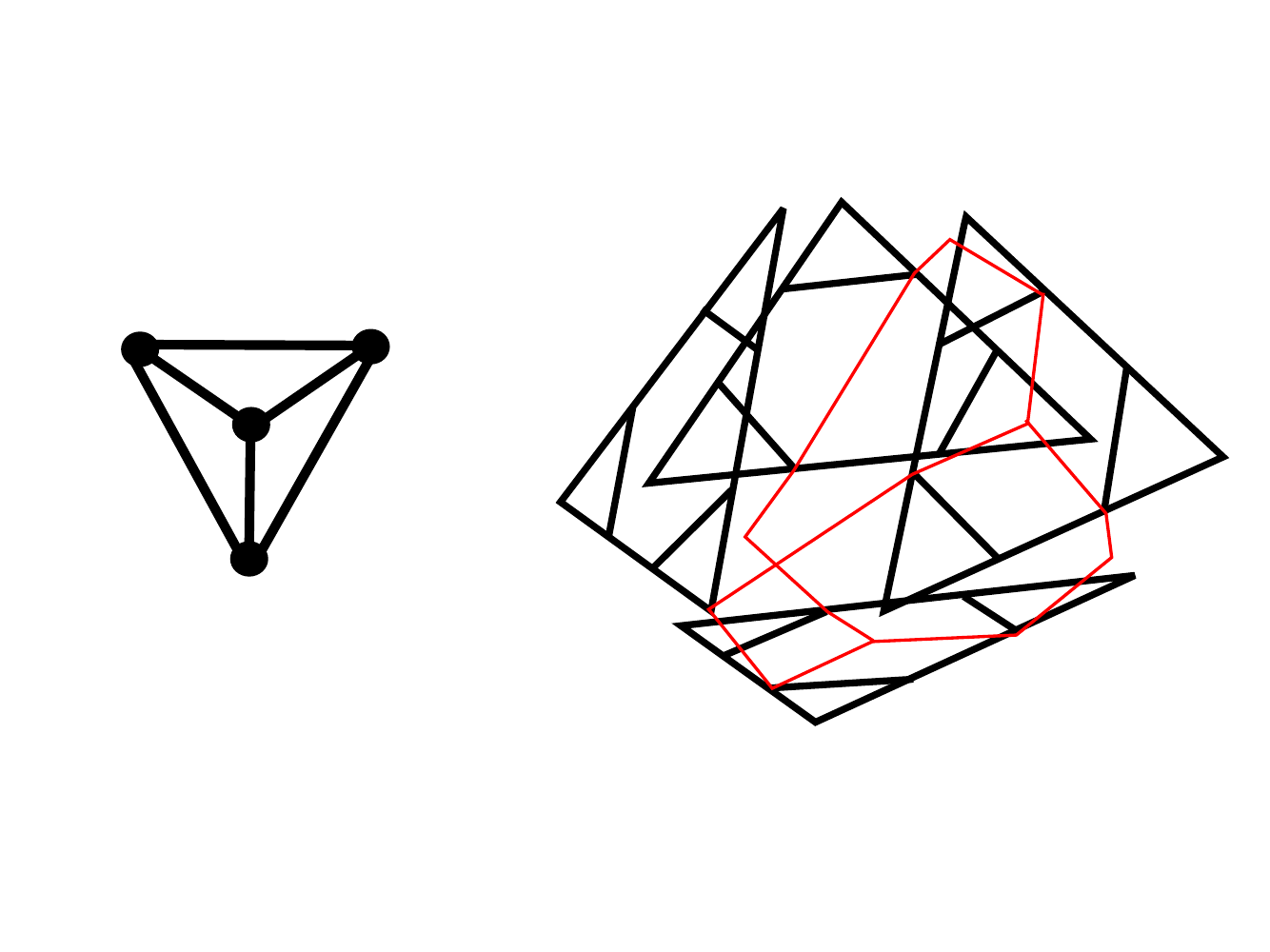}
\caption{A BZ quilt with dual genus $3$ graph.}
\label{quilt2}
\end{figure}

\noindent
 We represent gluing triangles $T_1, T_2$ with matching boundary components as weighted graphs on composite diagrams, as in Figure \ref{quilt}. Paths at the meeting boundaries of $T_1, T_2$ are joined by an arrangement of weighted paths in an $X$ configuration. When the path has weight $1$, this is either a line (see the left corner of the quilt in Figure \ref{quilt}) or by a crooked line (see the right corner of the quilt in Figure \ref{quilt}).  The cones $BZ_{\tree}(SL_3(\C))$ are studied by the author and Zhou in \cite{MZ}, see also \cite{KM} for a relationship between Berenstein-Zelevinsky triangles and mathematical biology.  

We finish this subsection with an analysis of the generators of the semigroup algebra $\C[BZ_{\Gamma}(SL_2(\C))]$.  Recall from Example \ref{sl2example} that the semigroup $BZ_3(SL_2(\C))$ is the free semigroup on three generators, we depict an element of this semigroup as an arrangment of three types of paths in a dual trinode $\tau$, see Figure \ref{graphweight}. 

\begin{figure}[htbp]
\centering
\includegraphics[scale = 0.24]{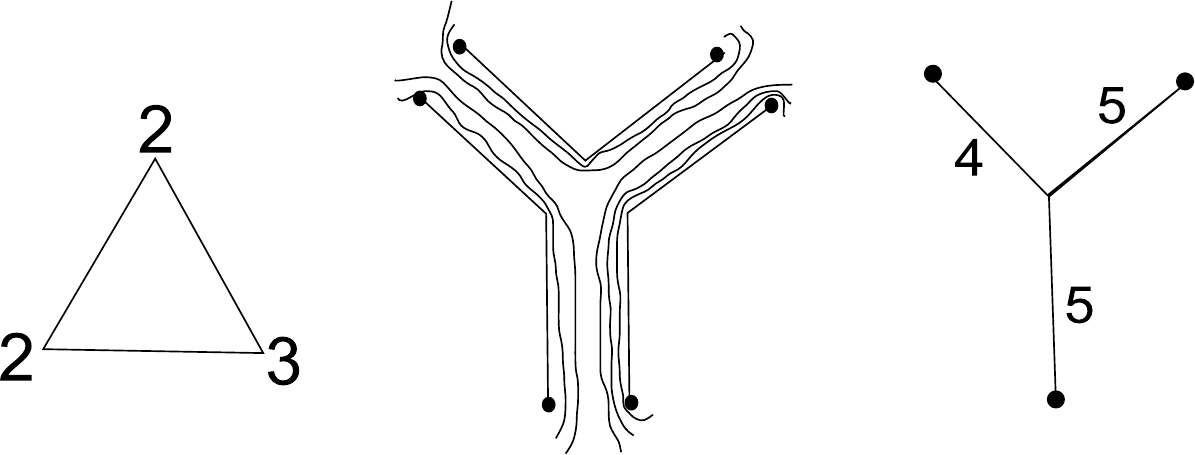}
\caption{Three different interpretations of an element of $BZ_3(SL_2(\C))$.}
\label{graphweight}
\end{figure}

For an element $T \in BZ_3(SL_2(\C))$, counting the number of endpoints in each edge $e, f, g$ of $\tau$ produces an integer weighting $w_T: \{e, f, g\} \to Z_{\geq 0}.$    We associate a planar arrangement of paths $P(w)$ in $\Gamma$, by replacing the weights at each trinode with paths as in Figure \ref{graphweight}. The endpoints of these paths in an edge $e$ are then connected in the unique planar way. A path $\gamma$ in the graph $\Gamma$ has an associated weighting $w_{\gamma}: E(\Gamma) \to Z_{\geq 0}$ obtained by setting $w_{\gamma}(e)$ equal to the number of times $\gamma$ passes through $e.$ For a more in depth account of this construction, see \cite{M16}.     

\begin{theorem}
The semigroup $BZ_{\Gamma}(SL_2(\C))$ is generated by the $w: E(\Gamma) \to \Z_{\geq 0}$
with $w(e) \leq 2$. 
\end{theorem}

\begin{figure}[htbp]
\centering
\includegraphics[scale = 0.3]{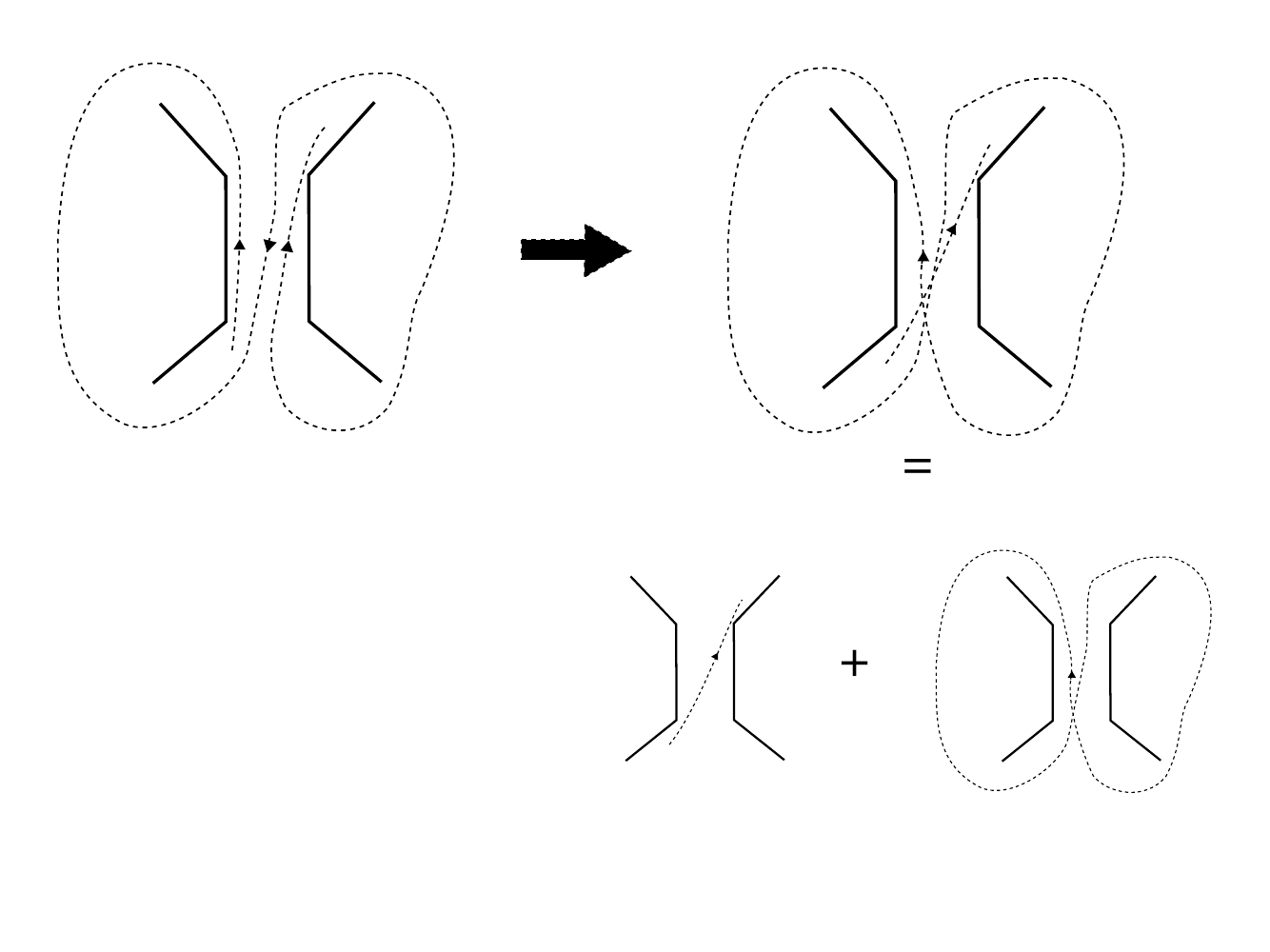}
\caption{Using an orientation to factor off a loop.}
\label{graphproof}
\end{figure}

\begin{proof}
Fix a $w \in BZ_{\Gamma}(SL_2(\C))$, and consider the induced planar arrangement
of paths $P(w)$ in $\Gamma$ with multiweight $w.$  Suppose $w(e) > 2$ for some edge
$e \in E(\Gamma).$ We pick a path $\gamma \in P(w)$ which passes through $e$. If $\gamma$
passes through $e$ with weight $1$, then we may remove $w_{\gamma}$ to obtain a weighting $w'$
with strictly smaller total weight $\sum_{f \in E(\Gamma)} w'(f).$  

If $\gamma$ weights $e$ greater than $1$, we assign an orientation to $\gamma.$   If two components at $e$ have the same direction, we may alter the weighting as in Figure \ref{graphproof}, yielding two closed paths $\gamma' \cup \gamma''.$ Without loss of generality we assume that $P(w) = \{\gamma\},$ so that the weightings $w_{\gamma'}$ and $w_{\gamma''}$ satisfy $w_{\gamma'} + w_{\gamma''} = w.$  In this case we pull off the new closed path $\gamma',$ which has strictly smaller total weight. If $w(e) > 2$ at least two components through $e$ must have the same direction.
\end{proof}

As a corollary, a set of regular functions in $\C[\mathcal{X}(F_g, SL_2(\C))]$ which represent the  $w \in BZ_{\Gamma}(SL_2(\C))$ with $w(e) \leq 2$ form a finite generating set.  Such a set is said to be a finite ``Khovanskii basis'' of the filtration defined by $\Gamma.$

\subsection{$G_2$}

We give inequalities for the tensor product cones for $G_2$, with $R(w_0) = \{\alpha_1\alpha_2\alpha_1\alpha_2\alpha_1\alpha_2,  \alpha_2\alpha_1\alpha_2\alpha_1\alpha_2\alpha_1\}.$ These cones have six string parameters  $t_1, t_2, t_3,$ $t_4, t_5, t_6,$ and six weight parameters $\lambda = (\lambda_1, \lambda_2), \eta = (\eta_1, \eta_2), \mu = (\mu_1, \mu_2)$.  The cone $C_{\alpha_1\alpha_2\alpha_1\alpha_2\alpha_1\alpha_2}(3)$ is defined by the following inequalities: 

\begin{equation}
6t_2 \geq 2t_3 \geq 3t_4 \geq 2t_5 \geq 6t_6 \geq 0; \ \  \lambda_2 \geq 2t_6\\
\end{equation}

\begin{equation}
\eta_1 \geq t_1 - 3t_2 - t_3 - 3t_4 - t_5 - 3t_6, \ \ t_3 - 3t_4 - t_5 - 3t_6, \ \ t_5 - 3t_6\\
\end{equation}

\begin{equation}
\eta_2 \geq t_6, \ \ t_4 - t_5 - 3t_6, \ \ t_2 - t_3 - 3t_4 - t_5 - 3t_6\\
\end{equation}

\begin{equation}
2t_1 - 3t_2 + 2t_3 -3t_4 + 2t_5 -3t_6 = \lambda_1 + \eta_1 - \mu_1\\
\end{equation}

$$ -t_1 +2t_2 -t_3 +2t_4 -t_5 +2t_6 = \lambda_2 + \eta_2 - \mu_2$$

The alternative cone $C_{\alpha_2\alpha_1\alpha_2\alpha_1\alpha_2\alpha_1}(3)$ is defined by the following inequalities: 

\begin{equation}
2t_2 \geq 2t_3 \geq t_4 \geq 2t_5 \geq 2t_6 \geq 0\\
\end{equation}

\begin{equation}
\lambda_1 \geq t_6; \ \ \lambda_2 \geq t_2 + t_4 - t_5, \ \ t_2 + t_5 - t_6,  \ \ t_3 - t_4 - t_6,  \ \ t_5 - 3t_6, \\ 
\end{equation}

$$t_2 + t_3 -2t_4,  \ \ 2t_2 - t_4, \ \ 3t_2 - t_3, \ \ t_3 - t_5, \ \ t_4 - 2t_6, \ \ 2t_4 - t_5 - t_6, \ \ 3t_4 - 2t_5$$

\begin{equation}
\eta_1 \geq t_6, \ \ t_4 -3t_5 - t_6, \ \ t_2 - 3t_3 - t_4 - 3t_5 - t_6\\
\end{equation}

\begin{equation}
\eta_2 \geq t_5 - t_6, \ \  t_3 - t_4 -3t_5- t_6,  \ \ t_1 - t_2 - 3t_3 - t_4 - 3t_5 - t_6\\
\end{equation}

\begin{equation}
2(t_2 + t_4 +t_6) - 3(t_1 + t_3 + t_5) = \lambda_1 + \eta_1 - \mu_1\\
\end{equation}

$$2(t_1 + t_3 + t_5) - (t_2 + t_4 + t_6) = \lambda_2 + \eta_2 - \mu_2$$

\subsection{$Sp_4$}

We give the inequalities for the $Sp_4$ tensor product cones corresponding to the decompositions $R(w_0) = \{\alpha_1\alpha_2\alpha_1\alpha_2, \alpha_2\alpha_1\alpha_2\alpha_1\}$.  The cone $C_{\alpha_1\alpha_2\alpha_1\alpha_2}(3)$ is defined by the following inequalities:

\begin{equation}
t_2 \geq t_3 \geq t_4,\\
\end{equation}

\begin{equation}
\lambda_2 \geq t_4, \ \ 2t_3 - t_2,\ \ t_2 - 2t_1; \ \ \lambda_1 \geq t_1\\
\end{equation}

\begin{equation}
\eta_2 \geq t_2 + 2t_4 - 2t_3, \ \ t_4; \ \ \eta_1 \geq t_1 + 2t_3 - 2t_2,  \ \ t_3 - t_2\\
\end{equation}

\begin{equation}
2t_1 -t_2 +2t_3 -t_4 = \lambda_1 +\eta_1 - \mu_1\\
\end{equation}

$$-2t_1 +2t_2 -2t_3 +2t_4 = \lambda_2 + \eta_2 - \mu_2$$

The cone $C_{\alpha_2\alpha_1\alpha_2\alpha_1}(3)$ is defined by the following inequalities:

\begin{equation}
t_2 \geq t_3 \geq t_4 \geq 0 \ \ ; \lambda_2 \geq 2t_1, t_2; \ \ \lambda_1 \geq 2t_1\\
\end{equation}

\begin{equation}
\eta_1 \geq t_3 -t_4, \ \ t_1 -t_2 + 2t_3; \ \ \eta_2 \geq t_2 - 2t_3 + 2t_4,  \ \ t_4\\
\end{equation}

\begin{equation}
t_1  + t_2 + t_3 + t_4 = \lambda_1 +\eta_1 - \mu_1; \ \ t_2 + t_4 - t_1 - t_3 = \lambda_2 + \eta_2 - \mu_2
\end{equation}

\bibliographystyle{alpha}
\bibliography{Biblio}

\bigskip
\noindent
Christopher Manon:\\
Department of Mathematics,\\ 
George Mason University,\\ 
Fairfax, VA 22030 USA

\end{document}